\documentclass[a4paper, 12pt]{amsart}

\usepackage{amsmath,amssymb,amsthm,url,scalerel}
\usepackage[utf8]{inputenc}
\usepackage[T1]{fontenc}
\usepackage{libertine}
\usepackage[libertine,cmintegrals,cmbraces]{newtxmath}
\usepackage{verbatim}
\usepackage[shortlabels]{enumitem}
\usepackage{stmaryrd}
\usepackage{mathtools}
\usepackage{microtype}
\usepackage{tabto}
\usepackage{xcolor}
\usepackage{euscript}
\usepackage[all]{xy}
\usepackage{tikz}
\usepackage{tikz-cd}
\usepackage{nicefrac}
\usepackage{tabularx}
\usepackage{dsfont}
\usepackage{todonotes}
\usepackage{a4wide}
\usepackage{extarrows}

\numberwithin{equation}{section} 

\AtBeginDocument{%
   \def\MR#1{}
}

\usepackage[pagebackref]{hyperref}
  
\hypersetup{%
  bookmarksnumbered=true,
  colorlinks=true,%
  linkcolor=blue,%
  citecolor=blue,%
  filecolor=blue,%
  menucolor=blue,%
  urlcolor=blue,%
  pdfnewwindow=true,%
  pdfstartview=FitBH}

\usepackage{xcolor}
\definecolor{seagreen}{RGB}{46,139,87}
\definecolor{maroon}{RGB}{128,0,0}
\definecolor{darkviolet}{RGB}{148,0,211}
\definecolor{twelve}{RGB}{100,100,170}
\definecolor{thirteen}{RGB}{100,150,50}
\definecolor{fourteen}{RGB}{200,0,0}
\definecolor{fifteen}{RGB}{0,200,0}
\definecolor{sixteen}{RGB}{0,0,200}
\definecolor{seventeen}{RGB}{200,0,200}
\definecolor{eighteen}{RGB}{0,200,200}

\newcommand{\R}{\Bbb{R}}
\newcommand{\C}{\Bbb{C}}
\newcommand{\Z}{\Bbb{Z}}
\newcommand{\U}{\mathrm{U}}
\newcommand{\F}{\Bbb{F}}

\newcommand{\cC}{\mathcal{C}}
\newcommand{\cD}{\mathcal{D}}
\newcommand{\E}{\mathcal{E}}
\newcommand{\I}{\mathrm{I}}
\newcommand{\J}{\mathcal{J}}

\newcommand{\V}{\mathcal{V}}

\DeclareMathOperator{\colim}{\mathrm{colim}}

\DeclareMathOperator{\holim}{\mathrm{holim}}

\newcommand{\bigslant}[2]{{\raisebox{.2em}{$#1$}\left/\raisebox{-.2em}{$#2$}\right.}}

\DeclareMathOperator{\homog}{--homog--}

\DeclareMathOperator{\poly}{--poly--}

\newcommand{\s}{\mathrm{Sp}}
\DeclareMathOperator{\T}{\mathrm{Top}_\ast}

\newcommand{\Aut}{\mathrm{Aut}}

\DeclareMathOperator{\Hom}{\mathrm{Hom}}

\DeclareMathOperator{\Map}{\mathrm{Map}}

\DeclareMathOperator{\nat}{\mathrm{nat}}

\newcommand{\Ev}{\mathrm{Ev}}

\DeclareMathOperator{\id}{\mathrm{Id}}
\DeclareMathOperator{\ind}{\mathrm{ind}}
\DeclareMathOperator{\res}{\mathrm{res}}

\DeclareMathOperator{\BO}{\mathrm{BO}}
\DeclareMathOperator{\BU}{\mathrm{BU}}

\renewcommand{\O}{\mathrm{O}}

\newcommand{\op}{\mathrm{op}}

\newcommand{\Tow}{\mathrm{Tow}}

\newcommand{\CJ}{\J^{\mathbf{R}}}
\newcommand{\UJ}{\J^{\mathbf{U}}}

\newcommand{\RE}{\E^\mathbf{R}}

\newcommand{\CE}{C_2 \ltimes \E^\mathbf{R}}
\newcommand{\os}{\mathrm{Sp}^\mathbf{O}}
\newcommand{\us}{\mathrm{Sp}^\mathbf{U}}

\newcommand{\bR}{\mathbf{R}}
\newcommand{\bU}{\mathbf{U}}
\newcommand{\bF}{\mathbf{F}} 

\newcommand{\IUR}{\I_\mathbf{U}^\mathbf{R}}
\newcommand{\IRU}{\I_\mathbf{R}^\mathbf{U}}

\theoremstyle{definition}
\newtheorem{thm}{Theorem}[section]
\newtheorem{prop}[thm]{Proposition}

\newtheorem{lem}[thm]{Lemma}
\newtheorem{cor}[thm]{Corollary}

\newtheorem{ex}[thm]{Example}
\newtheorem{examples}[thm]{Examples}
\newtheorem{definition}[thm]{Definition}
\newtheorem{rem}[thm]{Remark}

\newtheorem{alphtheorem}{Theorem}[section]

  \makeatletter
\newcommand*{\centerfloat}{%
  \parindent \z@
  \leftskip \z@ \@plus 1fil \@minus \textwidth
  \rightskip\leftskip
  \parfillskip \z@skip}
\makeatother

  \newcommand{\adjunction}[4]{
\xymatrix@C+2cm{
#1:#2 \ar@<1ex>[r] &
\ar@<1ex>[l] #3:#4
}}

\begin{document}

\title[Recovering unitary calculus from calculus with reality]{Recovering unitary calculus from calculus with reality}
\author{Niall Taggart}
\address{Mathematical Institute, Utrecht University}
\email{n.c.taggart@uu.nl}
\date{\today}
\subjclass[2010]{Primary: 55P65. Secondary: 55P42, 55P91, 55U35}
\keywords{Functor calculus, unitary calculus, calculus with reality, stable splitting, Stiefel manifolds}
\maketitle

\begin{abstract}
By analogy with complex $K$--theory and $K$--theory with reality, there are theories of unitary functor calculus and unitary functor calculus with reality, both of which are generalisations of Weiss' orthogonal calculus. In this paper we show that unitary functor calculus can be completely recovered from the unitary functor calculus with reality, in analogy to how complex topological $K$--theory is completely recovered from $K$--theory with reality via forgetting the $C_2$--action. 
\end{abstract}

\setcounter{tocdepth}{1}
{\hypersetup{linkcolor=black} \tableofcontents}

\section*{Introduction}

\subsection*{Background}
Orthogonal and unitary calculus are  homotopy theoretic tools for studying functors from the category of real (respectively complex) inner product spaces to the category of based topological spaces. The calculi have had far reaching applications including geometrically, in the study of Miller splittings and Mitchell--Ricther filtrations on (complex) Stiefel manifolds (see the work of Arone \cite{Ar01}) but also homotopy theoretic through the connection with Goodwillie calculus (see the work of Barnes-Eldred \cite{BE16} and the author \cite{Taggartthesis}) and hence the EHP--sequence (see the work of Behrens \cite{BehrensEHP}) used to pursue computations of the stable homotopy groups of spheres. These calculi are also closely related through the complexification--realification adjunction on the inner product space level, similar to comparisons between real and complex topological $K$--theory. As such, orthogonal and unitary calculus are in analogy with real and complex $K$-theory, respectively. This is strengthened by some of the functors which have been studied using orthogonal and unitary calculus, particularly the classifying spaces of orthogonal (respectively unitary) group functors, $\BO(-)$ and $\BU(-)$ (see work of Arone \cite{Ar02}). 

Complex conjugation of complex inner product spaces defines an action of $C_2 = \mathrm{Gal}(\C \vert \R)$, the cyclic group of order two, on the category of complex inner product spaces. In \cite{Ta20b}, the author constructed a version of unitary calculus to take into account this $C_2$--action by complex conjugation. This ``calculus with reality'' in beneficial for a number of reasons, most notably, the layers of the Taylor tower in calculus with reality are completely determined up to homotopy by spectra with an action of $C_2 \ltimes U(n)$, which is a finer classification than that produced in unitary calculus, where in the latter, the layers are determined by spectra with an action of $U(n)$. Moreover, this calculus with reality plays the role of $K$--theory with reality in the analogy between the calculi and $K$--theory. This paper, and a sequel (see , are an attempt to make this analogy precise. In particular, both real and complex $K$--theory may be completely recovered by $K$--theory with reality, via taking $C_2$--fixed points and forgetting the $C_2$--action, respectively (see the work of Atiyah \cite{At66}). In this paper, we show that unitary calculus may be completely recovered from the calculus with reality via forgetting the $C_2$--action on the calculus. We plan to examine the relationship between orthogonal calculus and calculus with reality in future work. By the work of the author in \cite{Ta20}, there is already a comparison between unitary calculus and orthogonal calculus, hence the combination of this paper and \cite{Ta20} give a comparison between the calculus with reality and orthogonal calculus, but this comparison is unsatisfactory for a number of reasons, and we are optimistic of a more streamlined approach.

\subsection*{The Approach}
Our approach here is two--fold. Unitary calculus is indexed on the category of finite--dimensional inner product subspaces of $\C^\infty$, whereas, to have a well--defined $C_2\T$--enrichment, calculus with reality is indexed on the category of complexified finite--dimensional real inner product spaces, i.e., finite--dimensional inner product subspaces of $\C \otimes \R^\infty$ of the form $\C \otimes V$ for $V$ a finite--dimensional inner product subspaces of $\R^\infty$. The first strand of our approach is to show that the inclusion of universes $\C \otimes \R^\infty \hookrightarrow \C^\infty$ induces an equivalence of theories between unitary calculus indexed on subspaces of $\C^\infty$, and unitary calculus indexed on subspaces of $\C \otimes \R^\infty$. This equivalence of theories is induced by change of universe functors, $\IUR$ and $\IRU$, similar to those used by Mandell and May in their study of equivariant orthogonal spectra and $S$--modules \cite[V.1, VI.2]{MM02}. The notation is suggestive, $\IUR$ is the change of universe functor from unitary calculus to calculus with reality, and $\IRU$ is the change of universe functor from calculus with reality to unitary calculus, which is both left and right adjoint to $\IUR$, and  an equivalence of categories between the input categories of the calculi. 

The other aspect of our approach is to encode forgetting the $C_2$--action as a functor, denoted $i^*$, between the calculi, and exhibit that unitary calculus indexed on $\C \otimes \R^\infty$ is completely recovered by forgetting the $C_2$--action built into the calculus with reality. As our approach is two--fold, each section will have a subsection for each fold of our approach, and a final subsection giving the complete picture.

\subsection*{Key results}
The calculi under consideration are both categorifications of Taylor's Theorem from differential calculus. Through this analogy, the Taylor series is replaced by a Taylor tower approximating an input functor. The goal of the calculi is to use the information about the Taylor tower to recover information about our input functor. In detail, given a functor $F$, of the appropriate type, the calculi product a sequence of polynomial approximations $\{T_nF\}_{n \in \mathbb{N}}$, together with maps $T_nF \to T_{n-1}F$, for all $n \geq 1$. These maps assemble into a tower, approximating $F$. By considering the homotopy fibres of the maps in the tower, we gain information about the tower itself, and from this, information about the functor under investigation. These constructions are all functorial, hence the Taylor tower itself is functorial. We denote the Taylor tower of $F$ by $\Tow(F)$ for convenience.  The main result of this paper is a comparison of the Taylor towers in unitary calculus and calculus with reality. 

\begin{alphtheorem}[Theorem \ref{thm: recover tower}]\label{main thm}
The image, under our comparisons, of the Taylor tower of a functor with reality, is naturally isomorphic to the unitary Taylor tower of the image of the functor with reality under our comparisons, that is, for every functor with reality, $F$, there is a natural isomorphism
\[
\IRU(i^*(\Tow(F))) \cong \Tow(\IRU(i^*F)).
\]
\end{alphtheorem}

The power of the calculus with reality is captured in this comparison. It allows for several computations to be produced out of one computation. This is highlighted in Section \ref{section:Recovering stable splittings of Arone, Miller and Mitchell--Richter} where we prove stable splitting results using the calculus with reality, and demonstrate how these stable splittings completely recover stable splittings of Arone, Miller and Mitchell--Richter.

\begin{alphtheorem}[Theorem \ref{recovering of splitting}]
Let $F$ be a functor with reality. Suppose there exists a filtration of $F$ by sub--functors $F_n$, such that $F_0 = \ast$, and for all $n \geq 1$, the functor
\[
V \otimes \C \longmapsto F_n(V \otimes \C)/F_{n-1}(V \otimes \C),
\]
is (up to natural levelwise weak equivalence) of the form
\[
V_\C \longmapsto (X_n \wedge S^{\C^n \otimes V\otimes \C})_{h\U(n)},
\]
where $X_n$ is a based space equipped with an action of $C_2 \ltimes \U(n)$. Then the functor $\IRU(i^*F)$ is a unitary functor which stably splits in the sense of Arone~\cite[Theorem 1.1]{Ar01}.
\end{alphtheorem}

In order to prove Theorem \ref{main thm} we work along the Taylor tower, first comparing the polynomial approximations. In particular, there is a natural isomorphism between the image of the $n$--polynomial approximation of $F$ under forgetting the $C_2$--action through $i^*$, and change of universe $\IRU$, and the $n$--polynomial approximation of $\IRU(i^*F)$. We denote by $T_n^\bU$ the $n$--polynomial approximation functor in unitary calculus, and $T_n^\bR$ the $n$--polynomial approximation functor in calculus with reality, see Definition \ref{def: poly approx}.

 \begin{alphtheorem}[Theorem \ref{thm: recover polys}]\label{thm B}
 For a functor $F$, with reality, there is a natural isomorphism between $\IRU (i^*(T_n^\mathbf{R}F))$, and $T_n^\mathbf{U}(\IRU (i^*F))$.
 \end{alphtheorem}

The Taylor towers are characterised by homotopy fibre sequences 
\[
D_n F \longrightarrow T_nF \longrightarrow T_{n-1}F
\]
which encode the `error' between the $(n-1)$--polynomial approximation and the $n$--polynomial approximation. Theorem \ref{main thm} follows from Theorem \ref{thm B} by careful checking that our comparison functors preserve these homotopy fibre sequences. 

The homotopy theory of homogeneous functors of degree $n$ is captured in a model structure constructed as a right Bousfield localisation of a left Bousfield localisation of the projective model structure on the category of input functors (see Proposition \ref{prop: homog model structure}). This model structure, and hence the homotopy theory of homogeneous functors of degree $n$, is Quillen equivalent (via a zig--zag) to the stable model structure on the category of spectra with a $G(n)$--action, where the group $G(n)$ depends on the calculus under consideration. In Section \ref{section: model cats} we show how the model structures of unitary calculus may be recovered from the model structures for calculus with reality, producing a comprehensive diagram of model structures (see Figure \ref{fig: model cats for reality and UC}) which, after passage to homotopy categories commutes up to isomorphism. In particular, we exhibit Quillen equivalences between the $n$--homogeneous model structure on unitary calculus indexed on $\C^\infty$, and unitary calculus indexed on $\C \otimes \R^\infty$ via the change of universe functors.

\begin{alphtheorem}[Theorem \ref{thm: QE of homog model structures}]
The adjoint pairs
\[
\xymatrix@C+2cm{
n\homog\RE_0 \ar@<3ex>[r]|{\IRU } \ar@<-3ex>[r]|{\IRU }  & n\homog\E_0^\mathbf{U} \ar[l]|{\IUR },
}
\]
are Quillen equivalences between the $n$--homogeneous model structures, where this diagram should be interpreted as saying $\IUR$ has both a left and right adjoint.
\end{alphtheorem}

We do not claim that each unitary functor comes from a unique functor with reality. Given a unitary functor $F$, say, on the universe $\C \otimes \R^\infty$, we could define a $C_2$--action in several different ways, for example, we could take the internal $C_2$--action coming from complex conjugation on the matrices, which defines a $C_2$--action map
\[
F(\C \otimes V) \longrightarrow F(\C \otimes V), \quad F(c \otimes v) \longmapsto F(\bar{c} \otimes v).
\]
Alternatively, we could further affix an external $C_2$--action, giving a $C_2$--action map
\[
F(\C \otimes V) \longrightarrow F(\C \otimes V), \quad F(c \otimes v) \longmapsto g\cdot F(\bar{c} \otimes v),
\]
for $g \in C_2\setminus \{e\}$. These functors will have different Taylor towers in the calculus with reality, but forget to the same Taylor tower in unitary calculus. For example, consider the functor $\mathbb{S} : \C \otimes V \mapsto S^{\C \otimes V}$. The first action described above induced the complex conjugation action on $S^{\C \oplus V}$, whereas the second action in the case that we let $g \in C_2 \setminus \{e\}$ act by complex conjugation yields the trivial action on $S^{\C \oplus V}$. The difference in Taylor towers is immediate: one has first derivative with non-trivial $C_2$-action while the other has first derivative with trivial $C_2$-action.

\subsection*{Organisation}
We give a detailed overview of the calculi in Section \ref{section: the calculi}. The comparisons start in Section \ref{section: input functors} with the categories of input functors. A ``calculus'' is built from two main ideas; derivatives and polynomials. In Section \ref{section: derivatives} we examine the relationship between ``differentiation'' in unitary calculus and calculus with reality. Polynomial functors are the subject of study in Section \ref{section: polynomials}.  Using our comparisons of polynomial functors our attention turns to homogeneous functors, and the comparisons of homogeneous functors are complete in Section \ref{section: homogeneous}. We can hence directly compare the Taylor towers in Section \ref{section: Taylor towers}. The final section, Section \ref{section: model cats} gives a complete comparison between the model categories for unitary calculus and for calculus with reality. In Section \ref{section:Recovering stable splittings of Arone, Miller and Mitchell--Richter} we provide an application of our comparisons by demonstrating that stable splitting results in unitary calculus can be recovered from analogous results in the calculus with reality. 
In Section \ref{section:Recovering stable splittings of Arone, Miller and Mitchell--Richter} we prove a ``with reality'' version of a stable splitting of Arone, which in turn yields ``with reality'' stable splittings analogous to those of Miller and Mitchell--Richter. We further show that Arone's stable splitting is the image of our stable splitting under our comparison functors.

\subsection*{Notation and Conventions}
Given an adjoint pair, the left adjoint will always be written on the top or to the left, unless otherwise stated. A diagram of the form 
\[
\xymatrix@C+2cm{
\mathcal{A} \ar@<3ex>[r]|{f_1} \ar@<-3ex>[r]|{f_3}  & \mathcal{B} \ar[l]|{f_2},
}
\]
should be interpreted as saying that $f_1$ is left adjoing to $f_2$ and $f_2$ is left adjoint to $f_3$. We say a diagram of adjoint pairs is commutative if the sub--diagram of right adjoints commutes up to natural isomorphism, or equivalently, if  the sub--diagram of left adjoints commutes up to natural isomorphism. In some cases we will have to consider diagrams of the form
\[
\xymatrix@C+3cm@R+2cm{
\mathcal{A} \ar@<1ex>[d]^{g_2} \ar@<3ex>[r]|{f_1} \ar@<-3ex>[r]|{f_3}	& \mathcal{B} \ar@<1ex>[d]^{g_4} \ar[l]|{f_2}	 \\
\mathcal{A}  \ar@<1ex>[u]^{g_1} \ar@<3ex>[r]|{f_1} \ar@<-3ex>[r]|{f_3}	& \mathcal{B}  \ar@<1ex>[u]^{g_3} \ar[l]|{f_2}	  \\
}
\]
which we say commutes if both the diagram in which $f_2$ is a left adjoint, and the diagram in which $f_2$ is a right adjoint, commute in the above sense. 

\subsection*{Acknowledgements}
The author is thankful to Alexander Campbell for a helpful discussion on the cofinality statement for enriched ends in Proposition \ref{prop: cofinality of ends}. We thank David Barnes for careful reading of an earlier draft of this work and the anonymous referee for simplifying several arguments and greatly enhancing this article.

\section{The calculi}\label{section: the calculi}
In this section we give an overview of the theory of orthogonal calculus, unitary calculus and calculus with reality. Let $\F$ denote either $\R$, $\C$ or $\C \otimes \R$ and let $G(n)$ denote either $\O(n)$, $\U(n)$  or the semi--direct product $C_2 \ltimes \U(n)$, respectively, where $C_2$ acts on $\U(n)$ by term--wise complex conjugation. The use of $\C \otimes \R$ is to suggest a ``with reality'' approach to the calculi. For full details of the theories, see work of Weiss, Barnes-Oman and the author \cite{We95, BO13, Ta19, Ta20b}. We will use bold face letters to distinguish between the calculi, $\mathbf{O}$ for orthogonal calculus, $\mathbf{U}$ for unitary calculus, $\mathbf{R}$ for calculus with reality, and $\mathbf{F}$ for the general case.

\subsection{Input Functors} We begin by describing the \emph{input categories}, that is, the category of functors to which each calculus applies. We will refer to such functors as \emph{input functors}. Let $\J$ be the category of finite--dimensional $\F$--inner product subspaces of $\F^\infty$, and $\F$--linear isometries. For $\F=\C \otimes \R$ we think of $\F^\infty$ as $\C \otimes \R^\infty$, and consider subspaces of the form $V_\C :=\C \otimes V$, for $V$ a subspace of $\R^\infty$. Denote by $\J_0$ the category with the same objects as $\J$ and morphism space $\J_0(U,V) = \J(U,V)_+$. The category $\J_0$ is $\T$--enriched; $\J(U,V)$ is the Stiefel manifold of $\dim_\F (U)$--frames in $V$. 
These categories are the indexing categories for the functors under consideration in the calculi.

\begin{examples}\hfill
\begin{enumerate}
\item In the case $\F=\R$, the category $\J$ is the indexing category for orthogonal calculus, and is denoted $\J^\mathbf{O}$.
\item In the case $\F =\C$, the category $\J$ is the indexing category for unitary calculus, and is denoted $\J^\mathbf{U}$.
\item In the case $\F = \C \otimes \R$, the category $\J$ is another model for unitary calculus, and is denoted $\CJ$ (see Section \ref{section: input functors}). 
\item In the case $\F=\C \otimes \R$, the category $\J$ is $C_2\mathrm{Top}$--enriched, with the non--identity element $g$ of $C_2$ acting by conjugation by complex conjugation, i.e. for $f \in \J(U,V)$, 
\[
(g\cdot f)(u) = g(f(gu)) = \overline{f(\overline{u})}.
\]
In this case, $\J$ is the indexing category for calculus with reality, and we denoted it by $\J^\mathbf{R}$. It is important to note that using the universe $\C \otimes \R^\infty$ is crucial for a $C_2\T$--enrichment and hence for a well-defined calculus with reality (see \cite[\S\S1.1, \S\S1.2]{Ta20b}).
 \end{enumerate}
\end{examples}

\begin{definition}\label{def: input functors}
Define $\E_0$ to be the category of $\T$--enriched functors from $\J_0$ to $\T$, and define $\CE_0$ to be the category of $C_2\T$--enriched functors from $\CJ_0$ to $C_2\T$. 
\end{definition}

To ease notation when considering all calculi at once, we will denote $\CE_0$ by $\E_0$. We give some examples of the categories $\E_0$ and the calculi for which they are the input functors.

\begin{examples}\hfill
\begin{enumerate}
\item For $\F=\R$, we denote the category $\E_0$ by $\E_0^\mathbf{O}$. This is the category of input functors in orthogonal calculus, studied by Weiss \cite{We95} and Barnes and Oman \cite{BO13}.
\item For $\F=\C$, we denote the category $\E_0$ by $\E_0^\mathbf{U}$. This is the category of input functors in unitary calculus, studied by the author in \cite{Ta19}.
\item For $\F=\C \otimes \R$, we denote the category $\E_0$ by $\E_0^\mathbf{R}$. This is the underlying non-equivariant category of input functors for calculus with reality, and will play an important role in our comparisons. 
\item The category $\CE_0$ is the category of input functors in calculus with reality, studied by the author in \cite{Ta20b}.
\end{enumerate}
\end{examples}

These input categories are categories of diagram spaces in the sense of Mandell-May-Schwede-Shipley \cite[Definition 1.1]{MMSS01}, hence they may be equipped with a projective model structure. 

\begin{prop}[{\cite[Theorem 6.5]{MMSS01}}]
There is a cellular, proper and topological model structure on the category $\E_0$, with the weak equivalences and fibrations the levelwise weak homotopy equivalences and levelwise Serre fibrations respectively. The generating (acyclic) cofibrations are of the form $\J_0(U,-) \wedge i$, for $U \in \J_0$, where $i$ is a generating (acyclic) cofibration in $\T$. 
\end{prop}

There is also a projective model structure in $\CE_0$, defined similarly, but with different generating (acyclic) cofibrations, transferred from the underlying model structure on $C_2\T$. This model structure on $C_2\T$ is sometimes called the coarse model structure, and has weak equivalences and fibrations defined to be weak homotopy equivalences and Serre fibrations, respectively. 

\begin{prop}[{\cite[Theorem 6.5]{MMSS01}}]
There is a cellular, proper and topological model structure on the category $\CE_0$, with the weak equivalences and fibrations the levelwise weak homotopy equivalences and levelwise Serre fibrations of the coarse model structure on $C_2\T$, respectively. The generating (acyclic) cofibrations are of the form $\J_0(U,-) \wedge (C_2)_+ \wedge  i$, for $U \in \J_0$, where $i$ is a generating (acyclic) cofibration in $\T$. 
\end{prop}

\subsection{Polynomial functors} We give a short overview of polynomial functors, for full details on these functors see \cite{We95, BO13, Ta19, Ta20b}. The basic idea of a polynomial functor, is that of a categorification of polynomial functions in differential calculus. Denote by $\mathbf{F}_{\leq n}$ the poset of non--zero finite--dimensional subspaces of $\F^n$. For $\F = \C \otimes \R$, we only consider subspaces of the form $\C \otimes V$, for $V$ a finite--dimensional inner product subspace of $\R^{n}$. 

\begin{definition}\label{def: polynomial}
A functor $F \in \E_0$ is \textit{polynomial of degree less than or equal $n$} or equivalently \textit{$n$--polynomial} if the canonical map
\[
F(V) \to \underset{U \in \mathbf{F}_{\leq n+1}}{\holim}F(U \oplus V) =: \tau_nF(V)
\]
is a weak homotopy equivalence for all $V \in \J_0$. 
\end{definition}

In the case of $\CE_0$, the homotopy limit inherits a $C_2$--action via the $C_2$--action on $F(V \oplus U)$ for all $U \in \mathbf{R}_{\leq n+1}$. Iterating $\tau_n$ yields the polynomial approximation functor, which is the universal $n$--polynomial functor up to homotopy, that is, any map from a functor $F$ to a $n$--polynomial functor $E$, will, up to homotopy, factor through the $n$--polynomial approximation of $F$.

\begin{definition}\label{def: poly approx}
The \textit{$n$--polynomial approximation}, $T_nF$, of a functor $F \in \E_0$ is defined to be the homotopy colimit of the sequential diagram
\[
\xymatrix{
F \ar[r]^{\rho} &  \tau_nF \ar[r]^{\rho} &  \tau_n^2F \ar[r]^{\rho} &  \tau_n^3F \ar[r]^{\rho} & \cdots.  
}
\]
\end{definition}

Since an $n$--polynomial functor is $(n+1)$--polynomial (see \cite[Proposition 5.4]{We95}) these polynomial approximation functors assemble into a Taylor tower 
\[
\xymatrix{
		&	&	F(V) \ar@/_1pc/[dl]	  \ar@/^1pc/[dr]  \ar@/^1.3pc/[drr]   			     	&			& \\
 \cdots \ar[r]_{r_{n+1}} & T_nF(V) \ar[r]_{r_n} & \cdots \ar[r]_{r_2} & T_1F(V) \ar[r]_{r_1} & F(\C^\infty) \\
 &  D_nF(V) \ar[u] & & D_1F(V) \ar[u] &\\
}
\]
approximating a given input functor. 

Moreover there is a model structure on $\E_0$ which captures the homotopy theory of $n$--polynomial functors, given by suitably left Bousfield localising the projective model structure on $\E_0$. 

\begin{prop}[{\cite[Proposition 6.5]{BO13}, \cite[Proposition 2.8]{Ta19}, \cite[Proposition 2.5]{Ta20b}}]
There is a cellular, proper and topological model structure on $\E_0$, where a map $f\colon E \to F$ is a weak equivalence if $T_nf\colon T_nE \to T_nF$ is a levelwise weak equivalence, the cofibrations are the cofibrations of the projective model structure and the fibrations are levelwise fibrations such that 
\[
\xymatrix{
E \ar[r]^f \ar[d]_{\eta_E} & F \ar[d]^{\eta_F} \\
T_nE \ar[r]_{T_nf} & T_n F
}
\]
is a homotopy pullback square. The fibrant objects of this model structure are precisely the $n$--polynomial functors and $T_n$ is a fibrant replacement functor. We call this the $n$--polynomial model structure and it is denoted $n\poly\E_0$.
\end{prop}

\subsection{Homogeneous functors} The $n$--th layer of the Taylor tower, i.e. the homotopy fibre of the map $T_nF \to T_{n-1}F$, denoted $D_nF$, satisfies the property that it is both $n$--polynomial, and has trivial $(n-1)$--polynomial approximation (see e.g. \cite[Example 2.10]{Ta19}). The class of functors which satisfy this property are called $n$--homogeneous. 

\begin{definition}
A functor $F \in \E_0$ is said to be \textit{homogeneous of degree $n$} or equivalently \textit{$n$--homogeneous} if it is both $n$--polynomial and has trivial $(n-1)$--polynomial approximation. 
\end{definition}

There is a further model structure on $\E_0$ which captures the homotopy theory of $n$--homogeneous functors, given by suitably right Bousfield localising the $n$--polynomial model structure. 

\begin{prop}[{\cite[Proposition 6.9]{BO13}, \cite[Proposition 3.13]{Ta19}, \cite[Proposition 3.2]{Ta20b}}]\label{prop: homog model structure}
There is a topological model structure on $\E_0$ where the weak equivalences are those maps $f$ such that $D_nf$ is a weak equivalence in $\E_0$, the fibrations are the fibrations of the $n$-polynomial model structure and the cofibrations are those maps with the left lifting property with respect to the acyclic fibrations. The fibrant objects are $n$-polynomial and the cofibrant--fibrant objects are the projectively cofibrant $n$-homogeneous functors.
\end{prop}

In \cite[\S8]{Ta19}, the author gave further characterisations of the $n$--homogeneous model structure. We start with the acyclic fibrations.

\begin{prop}[{\cite[Proposition 8.3, Corollary 8.4]{Ta19}}]\label{characterisation of acyclic fibs}\label{lem: acyclic fibrations in n-homog}\hfill
\begin{enumerate}
\item A map $f\colon E \to F$ is an acyclic fibration in the $n$--homogeneous model structure if and only if it is a fibration in the $(n-1)$--polynomial model structure and an $D_n$--equivalence.
\item A map $f\colon E \to F$ between $n$--polynomial objects is an acyclic fibration in the $n$--homogeneous model structure if and only if it is a fibration in the $(n-1)$--polynomial model structure.
\end{enumerate}
\end{prop}

We can also completely characterise the cofibrations, and cofibrant objects.

\begin{lem}[{\cite[Lemma 8.5, Corollary 8.5]{Ta19}}]\label{lem: cofibrations of n-homog}\label{cofibrant objects of n-homog}\hfill
\begin{enumerate}
\item A map $f\colon E \to F$ is a cofibration in the $n$--homogeneous model structure if and only if it is a projective cofibration and an $(n-1)$--polynomial equivalence. 
\item The cofibrant objects of the $n$--homogeneous model structure are precisely those $n$--reduced projectively cofibrant objects.
\end{enumerate}
\end{lem}

\subsection{The intermediate categories} Homogeneous functors of degree $n$ are characterised by spectra with an action of $G(n)$. More precisely, the $n$--homogeneous model structure on $\E_0$, and spectra with an action of $G(n)$ are Quillen equivalent via a zig--zag through (at least one) intermediate category. These intermediate categories are the natural home for the $n$-th derivative of a functor.  We give an overview of the construction of these intermediate categories and how they relate to spectra and the $n$--homogenous model structure. The length of the zig--zag is dependent on the version of calculi which one considers. We start with the details which are common to all three calculi. 

Sitting over the space of linear isometries $\J(U,V)$ is the the $n$--th complement vector bundle, with total space 
\[
\gamma_n(U,V) = \{ (f,x) \ : \ f \in \J(U,V), x \in \F^n \otimes_\F f(U)^\perp\}
\]
where we have identified the cokernel of $f$ with $f(U)^\perp$, the orthogonal compliment of $f(U)$ in $V$. This total space has an action of $\Aut(n) = \Aut(\F^n)$ via the regular representation. In the case $G(n) = C_2 \ltimes \U(n)$, the vector bundle $\gamma_n^\mathbf{R}(V,W)$ is a Real vector bundle in the sense of Atiyah \cite{At66}.

\begin{definition}\label{def: jet categories}
Define $\J_n$ to be the category with the same objects as $\J$ and morphism $G(n)$--space $\J_n(U,V)$ given by the Thom space of the vector bundle $\gamma_n(U,V)$. 
\end{definition}

 With this, we may define the intermediate categories. 

\begin{definition}\label{def: intermediate cat}
Define $\E_n$ to be the category of $\T$-enriched functors from $\J_n$ to $\T$, and define the $n$--th intermediate category $G(n)\E_n$, to be the category of $G(n)\T$--enriched functors from $\J_n$ to $G(n)\T$.
\end{definition}

Let $n\mathbb{S}$ be the functor given by $V \longmapsto S^{nV}$ where $nV := \F^n \otimes_\F V$. By \cite[Proposition 7.4]{BO13}, \cite[Proposition 4.2]{Ta19} and \cite[Proposition 4.13]{Ta20b} the intermediate categories are equivalent to the category of $n\mathbb{S}$--modules in $G(n)$--equivariant $\mathbf{F}$--spaces, where $\mathbf{F}$ is the category of finite--dimensional $\F$--inner product subspaces of $\F^\infty$ with $\F$--linear isometric isomorphisms. We equip $G(n)\E_n$ with a stable model structure, similar to the stable model structure on spectra, with weak equivalences given by $n\pi_*$-isomorphisms, i.e., $\pi_*$-isomorphisms altered to take into account the $n\mathbb{S}$-module structure. The $n$-homotopy groups for $X \in G(n)\E_n$ are defined as the colimit
\[
n\pi_k(X) = \underset{q}{\colim}~ \pi_{k+\dim(\F^n \otimes_\F \F^q)}X(\F^q),
\]
for all $k \in \Z$. For example, if $X \in \O(n)\E_n^\mathbf{O}$ then
\[
n\pi_k(X) = \underset{q}{\colim}~ \pi_{k+nq}X(\R^q),
\]
or, if $X \in \U(n)\E_n^\mathbf{U}$ the $n$-homotopy groups are given by,
\[
n\pi_k(X) = \underset{q}{\colim}~ \pi_{k+2nq}X(\C^q).
\]

\begin{prop}[{\cite[Proposition 7.4]{BO13}, \cite[Proposition 5.6]{Ta19}, \cite[Theorem 4.18]{Ta20b}}]\label{prop: n-stable model structure}
There is a cofibrantly generated, proper, and topological model structure on the category $G(n)\E_n$, where the weak equivalences are the $n\pi_*$-isomorphisms, and the fibrations are those levelwise fibrations $f\colon X \to Y$ such that the diagram
\[
\xymatrix{
X(V) \ar[r] \ar[d] & \Omega^{nW}X(V \oplus W) \ar[d] \\
Y(V) \ar[r] & \Omega^{nW}Y(V \oplus W),
}
\]
is a homotopy pullback square for all $V,W \in \J_n$. The fibrant objects of the $n$-stable model structure are called $n\Omega$-spectra and have the property that the adjoint structure maps,  $X(V) \to \Omega^{nW} X(V \oplus W)$, are levelwise weak equivalences for all $V,W \in \J_n$. We call this the $n$--stable model structure. 
\end{prop}

In each version of the calculus, the intermediate category is Quillen equivalent to the category of orthogonal spectra with an action of $G(n)$. These Quillen equivalences are constructed differently for each of the calculi, however, the first step in the zig--zag is always similar. The first Quillen equivalence relates the intermediate category to the category $\E_1[G(n)]$. In the orthogonal and unitary cases this reduces to a Quillen equivalence between the intermediate category and the category of spectra with a $G(n)$--action, as the categories of spectra and $\E_1$ are equivalent, they are both $\mathbb{S}$--modules in the category of $\mathbf{F}$--spaces. 

There is a topological functor $\alpha_n \colon \J_n \to \J_1$, given by $\alpha_n(V) = \F^n \otimes_\F V$ on objects, and on morphism spaces
\[
\begin{split}
(\alpha_n)_{U,V} \colon \J_n(U,V) &\longrightarrow \J_1(nU, nV) \\
(f,x) &\longmapsto (\F^n \otimes_\F f, x).
\end{split}
\]
Given an object $\Theta$ in $\E_1[G(n)]$, the space $\Theta(nU) = \Theta(\F^n \otimes_\F U)$ has an internal $G(n)$--action given by the action of $G(n)$ on $\F^n$. For $g \in G(n)$, we denote this action by the self map 
\[
\Theta(g \otimes U) \colon \Theta(nU) \longrightarrow \Theta(nU).
\] 
The space $\Theta(nU)$ also has an external action given by the fact that for any $V \in \J$, $\Theta(V)$ is a $G(n)$--space. For $g \in G(n)$, we denote this action by the self map 
\[
g_{\Theta(nU)} \colon \Theta(nU) \longrightarrow \Theta(nU).
\] 
By letting $g \in G(n)$ act on $(\alpha_n)^*\Theta(U) = \Theta(nU)$ by 
\[
\Theta(g \otimes U) \circ g_{\Theta(nU)},
\]
precomposition with $\alpha_n$ defines a functor 
\[
(\alpha_n)^* \colon \E_1[G(n)] \longrightarrow G(n)\E_n,
\]
which has a left adjoint given by the enriched left Kan extension along $\alpha_n$, which can be described as the coend,
\[
((\alpha_n)_! X)(V) = \int^{U \in \J_n} X(U) \wedge \J_1(nU, V).
\]

It is convenient to have different notation for $\alpha_n$ in each of the calculi. In orthogonal calculus we denote it by $\beta_n$. 

\begin{prop}[{\cite[Theorem 8.3]{BO13}}]
There is an adjoint pair
 \[
\xymatrix@C+2cm{
\O(n)\E_n^\mathbf{O}
\ar@<+1ex>[r]^{(\beta_n)_!}
&
\s^\mathbf{O}[\O(n)]
\ar@<+1ex>[l]^{(\beta_n)^*},
}
\]
which is a Quillen equivalence between the $n$--stable model structure on $\O(n)\E_n^\mathbf{O}$ and the stable model structure on $\s^\mathbf{O}[\O(n)]$. 
\end{prop}

The realification functor $r \colon \J_0^\mathbf{U} \to \J_0^\mathbf{O}$, which forgets the complex structure induces, by precomposition, a right Quillen functor between orthogonal and unitary spectra. This lifts to a right Quillen functor between orthogonal spectra with a $G(n)$--action, $\os[G(n)]$, and unitary spectra with a $G(n)$--action, $\us[G(n)]$, (see \cite[Corollary 6.5]{Ta19}). In unitary calculus we combine this Quillen pair with the Quillen pair relating the intermediate category and unitary spectra with an action of $\U(n)$ to produce the follow sequence of Quillen equivalences, which may be composed to yield a single Quillen equivalence as composition of left (resp. right) Quillen functors produces a left (resp. right) Quillen functor. 

\begin{prop}[{\cite[Theorem 5.8]{Ta19}}]\label{QE for UC}
There is a sequence of adjoint pairs
 \[
\xymatrix@C+2cm{
\U(n)\E_n^\mathbf{U}
\ar@<+1ex>[r]^{(\alpha_n)_!}
&
\s^\mathbf{U}[\U(n)]
\ar@<+1ex>[l]^{(\alpha_n)^*}
\ar@<+1ex>[r]^{r_!}
&
\s^\mathbf{O}[U(n)]
\ar@<+1ex>[l]^{r^*},
}
\]
which yields a Quillen equivalence between the $n$--stable model structure on $\U(n)\E_n^\mathbf{U}$ and the stable model structure on $\s^\mathbf{O}[\U(n)]$.
\end{prop}

The added $C_2$--equivariance involved in the calculus with reality makes it convenient to have an extra step in the zig--zag of Quillen equivalences, which does not compose to give a single Quillen pair. The first step in the Quillen equivalence is similar to the orthogonal case, and relates the intermediate category and the category of Real $\U(n)$--spectra, denoted, $\CE_1[\U(n)]$, see \cite[\S\S5.2]{Ta20b}.

The second step in the zig--zag is motivated by Schwede's description of Real spectra, and the relationship between Real spectra and orthogonal spectra, \cite[Example 7.11]{Sch19}. There is a $C_2$-equivariant decomposition $\C^n \cong \R^n \oplus i\R^n$, where $i$ denotes the imaginary unit. It follows that $V_\C = V \otimes \C$ may be decomposed up to natural isomorphism as $V \oplus iV$, where $C_2$ acts on $iV$ via the sign representation. We denote by $S^{iV}$ the one-point compactification of $iV$, with the induced $C_2$-action. There is a functor $\psi : \CE_1[\U(n)] \to \os[C_2 \ltimes \U(n)]$ is given by 
\[
\psi(X)(V) = \Map_*(S^{iV}, X(V_\C)),
\]
which has left adjoint $L_\psi: \os[C_2 \ltimes \U(n)] \to \CE_1[\U(n)]$, given by the enriched coend 
\[
L_\psi(Y)(V_\C) =  \int^{U \in \J_1^\mathbf{O}} \CJ_1(U_\C, V_\C) \wedge Y(U) \wedge S^{iU}.
\]

\begin{prop}[{\cite[Theorem 5.2, Theorem 5.8]{Ta20b}}]
There is a sequence of adjoint pairs,
\[
\xymatrix@C+2cm{
C_2 \ltimes \U(n)\E_n^\mathbf{R}
\ar@<+1ex>[r]^{(\xi_n)_!}
&
\CE_1[\U(n)]
\ar@<+1ex>[l]^{(\xi_n)^*}
\ar@<-1ex>[r]_{\psi}
&
\s^\mathbf{O}[C_2 \ltimes \U(n)]
\ar@<-1ex>[l]_{L_\psi},
}
\]
which yields a zig--zag of Quillen equivalences between the $n$--stable model structure on the intermediate category $C_2 \ltimes \U(n)\E_n^\mathbf{R}$, and the stable model structure on $\s^\mathbf{O}[C_2 \ltimes \U(n)]$.
\end{prop}

\subsection{The derivatives of a functor} We now move on to discussing the derivatives of a functor. The derivatives are naturally objects in $G(n)\E_n$. Their definition comes from constructing an adjunction between $\E_0$ and $G(n)\E_n$. The inclusion $\F^m \to \F^n$ onto the first $m$-coordinates induces an inclusion of categories, $i_m^n \colon \J_m \to \J_n$. 

\begin{definition}\label{def: derivative}
Define the \textit{restriction functor} $\res_0^n \colon \E_n \to \E_m$ to be precomposition with $i_m^n$, and define the \textit{induction functor} $\ind_m^n \colon \E_m \to \E_n$ to be the right Kan extension along $i_m^n$. In the case $m=0$, the induction functor $\ind_0^n$ is called the \textit{$n$-th derivative}. 
\end{definition}

Combining this adjunction with the change of group functors from \cite[V.2]{MM02} provides an adjunction
\[
\adjunction{\res_0^n/\Aut(n)}{G(n)\E_n}{\E_0}{\ind_0^n \varepsilon^*}.
\]

This adjunction is a Quillen equivalence between the $n$-homogeneous model structure on $\E_0$ and the $n$-stable model structure on $G(n)\E_n$. 
\begin{prop}[{\cite[Theorem 10.1]{BO13}, \cite[Theorem 6.5]{Ta19}, \cite[Theorem 6.8]{Ta20b}}]
The adjoint pair
\[
\adjunction{\res_0^n/\Aut(n)}{G(n)\E_n}{n\homog\E_0}{\ind_0^n\varepsilon^*}
\]
is a Quillen equivalence between the $n$--stable model structure on $G(n)\E_n$ and the $n$--homogeneous model structure on $\E_0$. 
\end{prop}

\subsection{Classification of $n$-homogeneous functors}\label{subsect: class of homog}
For a functor $F \in \E_0$, denote by $\Psi_F^n$ the spectrum with an action of $G(n)$, which is the derived image of $F$ under the zig--zag of Quillen equivalences between the $n$--homogeneous model structure on $\E_0$ and spectra with an action of $G(n)$. These zig--zags are slightly different for each version of the calculi. For the orthogonal calculus we achieve the following
\[
\xymatrix@C+2cm{
n \homog \E_0^\mathbf{O}
\ar@<-1 ex>[r]_(0.6){\ind_0^n \varepsilon^*}
&
\O(n)\E_n^\mathbf{O}
\ar@<-1ex>[l]_(0.4){\res_0^n/\O(n)}
\ar@<+1ex>[r]^{(\beta_n)_!}
&
\s^\mathbf{O}[\O(n)]
\ar@<+1ex>[l]^{(\beta_n)^*}.
}
\]
For unitary calculus, the zig--zag is given as follows
\[
\xymatrix@C+2cm{
n \homog \E_0^\mathbf{U}
\ar@<-1 ex>[r]_(0.6){\ind_0^n \varepsilon^*}
&
\U(n)\E_n^\mathbf{U}
\ar@<-1ex>[l]_(0.4){\res_0^n/\U(n)}
\ar@<+1ex>[r]^{(\alpha_n \circ r)_!}
&
\us[\U(n)]
\ar@<+1ex>[l]^{(\alpha_n \circ r)^*}.
}
\]
Finally, the calculus with reality is given as follows
\[
\xymatrix@C+.55cm{
n \homog \CE_0
\ar@<-1 ex>[r]_(0.6){\ind_0^n \varepsilon^*}
&
C_2 \ltimes \U(n)\E_n^\mathbf{R}
\ar@<-1ex>[l]_(0.4){\res_0^n/\U(n)}
\ar@<+1ex>[r]^{(\xi_n)_!}
&
\CE_1[\U(n)]
\ar@<+1ex>[l]^{(\xi_n)^*}
\ar@<-1ex>[r]_{\psi}
&
\s^\mathbf{O}[C_2 \ltimes \U(n)]
\ar@<-1ex>[l]_{L_\psi}.
}
\]

\begin{prop}[{\cite[Theorem 7.3]{We95},\cite[Theorem 7.1]{Ta19}, \cite[Theorem 7.1]{Ta20b}}]\label{homog characterisation}
If $F \in \E_0$ is $n$-homogeneous for some $n >0$, then $F$ is levelwise weakly equivalent to the functor defined as 
\begin{equation*}\label{char of homog functors}
U \longmapsto \Omega^\infty [(S^{nU} \wedge \Psi_F^n)_{h\Aut(n)}],
\end{equation*}
where $\Aut(n) = \Aut_\J(\F^n)$, and $\Psi_F^n$ is a spectrum with an action of $G(n)$ defined as the derived image of $F$ under the zig-zag of Quillen equivalences. Conversely, any functor of this form is $n$--homogeneous.
\end{prop}

The characterisation allows for a different description of the weak equivalences in the $n$--homogeneous model structure. Such a description will be useful for the recovery of unitary calculus from the calculus with reality.

\begin{prop}[{\cite[Proposition 8.2]{Ta19}}]\label{prop: we in n--homog}
A map $f \colon E \to F$ in $\E_0$ is a $D_n$--equivalence if and only if 
\[
\ind_0^n T_nf \colon \ind_0^n T_n E \longrightarrow \ind_0^n T_n F,
\]
is a levelwise weak equivalence. 
\end{prop}

\section{Recovering the input functors}\label{section: input functors}

Our recovery efforts begin with the input functors for the calculi. We show that starting with a functor with reality, we can produce a unitary functor indexed on the universe $\C \otimes \R^\infty$, and further, every unitary functor is (up to equivalence of categories) of this form.

\subsection{Forgetting the $C_2$--action}

Forgetting the $C_2$--action on $\CE_0$ through the inclusion of the trivial subgroup $i: \{e\} \to C_2$, defines a functor 
\[
i^* : \CE_0 \longrightarrow \RE_0, F \longmapsto i^* F,
\]
where $(i^*F)(V) = i^*F(V)$ is the underlying space of $F(V)$. This functor has a left adjoint given by the free--construction
\[
\xymatrix@C+2cm{
\CE_0 \ar@<-1ex>[r]_{i^*} & \RE_0 \ar@<-1ex>[l]_{(C_2)_+ \wedge (-)}. \\
}
\]

The category $\RE_0$ can be viewed as an intermediate category between the input category of calculus with reality and the input category for unitary calculus. We shall see throughout this paper that one can build unitary calculus on $\RE_0$, and that this theory is equivalent to the unitary calculus built on $\E_0^\bU$. The above adjunction forms one half of the relationship between $\CE_0$ and $\E_0^\bU$, and moreover, the above adjunction is a Quillen adjunction with respect to the projective model structures.

\begin{prop}\label{prop: QA for forget on input}
The adjoint pair
\[
\xymatrix@C+2cm{
\CE_0 \ar@<-1ex>[r]_{i^*} & \RE_0 \ar@<-1ex>[l]_{(C_2)_+ \wedge (-)}, \\
}
\]
is Quillen adjunction between the projective model structure on $\CE_0$ and the projective model structure on $\RE_0$.
\end{prop}
\begin{proof}
As we are using the underlying model structure on $\CE_0$, the forgetful functor $i^*$ preserves levelwise weak equivalence and levelwise fibrations, hence the forgetful functor is a right Quillen functor. 
\end{proof}

\subsection{Change of universe}

Forgetting the $C_2$-action does not result in a unitary functor, it only permits functors evaluated on complexified real inner product spaces. To recover all of unitary calculus, we use the inclusion of the universe $\C \otimes \R^\infty \hookrightarrow \C^\infty$ to construct change of universe functors between $\RE_0$ and $\E_0^\mathbf{U}$, which will give an equivalence between the input functors for unitary calculus indexed on $\C^\infty$ and the input functors for unitary calculus indexed on $\C \otimes \R^\infty$. On the level of vector spaces, the inclusion of categories $j : \CJ_0 \to \UJ_0$, is a strong symmetric monoidal functor, hence, as in \cite[V.1.3]{MM02}, we can define a change of universe functor 
\[
\IUR \colon \E_0^\mathbf{U} \longrightarrow \RE_0,
\]
given by $(\IUR F)(V_\C) = F(j(V_\C))$. We can define (as in \cite[V.1.2]{MM02}) a change of universe functor in the other direction
\[
\IRU \colon \RE_0\longrightarrow \E_0^\mathbf{U}, 
\]
given by $\IRU  E(V) = \UJ_0(\C^n, V)\wedge_{\U(n)} E(\C \otimes \R^n)$, whenever $\dim_\C(V) = n$. These are inverse equivalences of categories by \cite[V.1.5]{MM02}.

\begin{prop}[{\cite[V.1.5]{MM02}}]
There is a diagram of adjoint functors
\[
\xymatrix@C+2cm{
\RE_0 \ar@<3ex>[r]|{\IRU } \ar@<-3ex>[r]|{\IRU }  & \E_0^\mathbf{U} \ar[l]|{\IUR },
}
\]
such that each adjoint pair defines an equivalence of categories. 
\end{prop}

These are also Quillen adjunctions between the projective model structures, as the functor $\IUR $ is both  left and right Quillen; it preserves levelwise weak equivalences, levelwise fibrations and generating cofibrations, hence also preserves projective cofibrations. 

\begin{prop}[{\cite[V.1.6]{MM02}}]\label{prop: QA for R and U}
The adjoint functors
\[
\xymatrix@C+2cm{
\RE_0 \ar@<3ex>[r]|{\IRU } \ar@<-3ex>[r]|{\IRU }  & \E_0^\mathbf{U} \ar[l]|{\IUR },
}
\]
define Quillen adjoint pairs between the projective model structures. 
\end{prop}

\subsection{Complete picture}

We have seen that unitary calculus defined on $\C^\infty$ and unitary calculus defined on $\C \otimes \R^\infty$ have equivalent input categories. On the homotopical side, the following result is a combination of Proposition \ref{prop: QA for forget on input} and Proposition \ref{prop: QA for R and U}, and summarises this section.

\begin{prop}
The adjoint functors
\[
\xymatrix@C+2cm{
\CE_0 \ar@<-1.5ex>[r]_{i^*} & \ar@<-1.5ex>[l]_{(C_2)_+ \wedge (-)} \RE_0 \ar@<3ex>[r]|{\IRU } \ar@<-3ex>[r]|{\IRU }  & \E_0^\mathbf{U} \ar[l]|{\IUR },
}
\]
define Quillen adjoint pairs between the projective model structures. 
\end{prop}

\begin{rem}
Since the category of input functors for unitary calculus is equivalent to the category $\RE_0$, unitary calculus may be constructed using $\RE_0$ in place of $\E_0^\mathbf{U}$. In such an instance, the unitary functors may be equipped with a $C_2$--action via complex conjugation on the vector space level, and hence be considered as objects of $\CE_0$. As such, the input functors for unitary calculus are completely recovered  (up to equivalence of categories) by forgetting the $C_2$--action on a functor with reality. 
\end{rem}

\section{Recovering the derivatives}\label{section: derivatives} 

One of the key components of any theory of calculus is that of derivatives. We start by considering how the derivatives of functors with reality relate to the derivatives of unitary functors. This is a significantly different approach than previously applied to comparing calculi; e.g. both Barnes and Eldred \cite{BE16} and the author \cite{Ta20} compare calculi by first comparing the homogeneous functors via their characterisations as infinite loop spaces.

\subsection{Forgetting the $C_2$--action}

We have seen that forgetting the $C_2$--action defines a functor $i^* \colon \CE_0 \longrightarrow \RE_0.$ In fact, this further defines a functor $i^*: \CE_n \to \RE_n$ for all $n \geq 0$.

\begin{lem}\label{lem: derivatives and forget}
For $F \in \CE_m$, there is a natural isomorphism $\ind_m^n (i^*F) \cong i^* (\ind_m^n F)$. That is, for $E \in \CE_0$, the $n$--th derivative of $i^*E \in \RE_0$ is naturally isomorphic to the image of the $n$--th derivative of $E$ under the forgetful functor $i^* : \CE_n \to \RE_n$. 
\end{lem}
\begin{proof}
The space $\ind_m^n (i^*F)(V)$ is given by the space of $\T$--enriched natural transformations $\RE_m(\CJ_n(V,-), i^*F)$, which can be expressed as the $\T$-enriched end
\[
\int_{W \in \CJ_m} \Map(\CJ_n(V,W) , i^*F(W)).
\]
On the other hand, at $V \in \CJ_m$, $\ind_m^n F(V)$ is given by the $C_2$--space of natural transformations 
\[
\CE_m(\CJ_n(V,-), F),
\]
which is the $C_2\T$-enriched end
\[
\int_{W \in \CJ_m} \Map(\CJ_n(V,W) , F(W)),
\]
where $C_2$ acts by conjugation. Applying the forgetful functor gives a $\T$-enriched end
\[
\int_{W \in \CJ_m} i^*\Map(\CJ_n(V,W) , F(W)),
\]
which is clearly homeomorphic to 
\[
\int_{W \in \CJ_m} \Map(\CJ_n(V,W) , i^*F(W)). \qedhere
\]
\end{proof}

\subsection{Change of universe}
To see how the change of universe functors interact with the derivatives, we use the description of the derivative as a right Kan extension along an inclusion of categories (see Definition \ref{def: derivative}). In particular, this describes the derivative of a functor as a right adjoint to a restriction functor, which in turn is defined by precomposition with the inclusion of categories. It is this adjoint pair which we consider with the change of universes.

\begin{lem}\label{lem: restriction and change of universe}
The diagram of adjoint pairs
\[
\xymatrix@C+1cm@R+1cm{
\RE_m \ar@<-1ex>[r]_{\ind_m^n}  \ar@<1ex>[d]^{\IRU }  &  \ar@<-1ex>[l]_{\res_m^n}  \ar@<1ex>[d]^{\IRU }  \RE_n \\
\E_m^\mathbf{U} \ar@<-1ex>[r]_{\ind_m^n}  \ar@<1ex>[u]^{\IUR }  &  \ar@<-1ex>[l]_{\res_m^n}   \ar@<1ex>[u]^{\IUR } \E_n^\mathbf{U} \\
}
\]
commutes up to natural isomorphism.
\end{lem}
\begin{proof}
We show that the left adjoints commute up to natural isomorphism. Consider the diagram (see Definition \ref{def: derivative})
\[
\xymatrix{
\CJ_m \ar[r]^{i_m^n} \ar[d]_{j} & \CJ_n \ar[d]^{j} \\
\J_m^\mathbf{U}  \ar[r]_{i_m^n} & \J_n^\mathbf{U}
}
\]
which is commutative on both objects and morphisms. It follows that there is a natural isomorphism 
\[
j^* \circ (i_m^n)^* = (i_m^n \circ j)^* \cong (j \circ i_m^n)^* = (i_m^n)^* \circ j^*,
\]
and hence by definition a natural isomorphism
\[
\IUR  \circ \res_m^n \cong \res_m^n \circ \IUR . \qedhere
\]
\end{proof}

As corollaries, we see that change of universe functors commute with derivatives in both directions. 

\begin{cor}\label{cor: derivatives and change of universe}
If $F \in \RE_m$, then there is a natural isomorphism $\ind_m^n (\IRU  F) \cong \IRU  (\ind_m^n F)$. In particular, for $E \in \RE_0$, the $n$-th derivative of $\IRU  E$ is naturally isomorphic to the image of the $n$-th derivative of $E$ under the change of universe functor $\IRU $.
\end{cor}
\begin{proof}
As the left adjoints in Lemma \ref{lem: restriction and change of universe} commute, so two must the right adjoints, and the result follows.
\end{proof}

In Lemma~\ref{lem: restriction and change of universe} we considered a commutative diagram relating change of universe functors and differentiation in which $\IUR$ acted as the left adjoint. In what follows we consider the analogous diagram in which $\IUR$ is acting as the right adjoint.

\begin{cor}\label{cor: derivatives commute with change of universe}
If $F \in \E_m^\mathbf{U}$, then there is a natural isomorphism $\ind_m^n (\IUR  F) \cong \IUR  (\ind_m^n F)$. In particular, for $E \in \E_0^\bU$, the $n$-th derivative of $\IUR  E$ is naturally isomorphic to the image of the $n$-th derivative of $E$ under the change of universe functor $\IUR $.
\end{cor}
\begin{proof}
We show that the diagram
\[
\xymatrix@C+1cm@R+1cm{
\RE_m \ar@<-1ex>[r]_{\ind_m^n}  \ar@<-1ex>[d]_{\IRU }  &  \ar@<-1ex>[l]_{\res_m^n}  \ar@<-1ex>[d]_{\IRU }  \RE_n \\
\E_m^\mathbf{U} \ar@<-1ex>[r]_{\ind_m^n}  \ar@<-1ex>[u]_{\IUR }  &  \ar@<-1ex>[l]_{\res_m^n}   \ar@<-1ex>[u]_{\IUR } \E_n^\mathbf{U}. \\
}
\]
commutes up to natural isomorphism by considering the left adjoints, and the result will follow from the right adjoints commuting up to natural isomorphism. By Lemma \ref{lem: restriction and change of universe}, there is a natural isomorphism $\IUR  \circ \res_m^n \cong \res_m^n \circ \IUR $. Postcomposition with $\IRU $ gives a natural isomorphism 
\[
\res_m^n \cong \IRU  \circ \IUR  \circ \res_m^n \cong \IRU  \circ \res_m^n \circ \IUR ,
\]
since $\IRU $ is the inverse equivalence of categories to $\IUR $. Precomposition with $\IRU $ further produces a natural isomorphism
\[
\res_m^n \circ \IRU  \cong  \IRU  \circ  \res_m^n \circ \IUR  \circ \IRU  \cong  \IRU   \circ \res_m^n,
\]
hence the left adjoints in the above diagram commute up to natural isomorphism.
\end{proof}

\subsection{Complete picture}

We gather together the process of forgetting the $C_2$--action and the change of universe to achieve the following result relating derivatives in calculus with reality and derivatives in unitary calculus. 

\begin{prop}\label{thm: recover derivatives}
If $F \in \CE_m$, then there is a natural isomorphism 
\[
\ind_m^n(\IRU  (i^*F)) \cong \IRU (i^*(\ind_m^n F)),
\]
that is, for $E \in \RE_0$, the $n$--th derivative of $\IRU  (i^*E)$ is naturally isomorphic to the image of the $n$--th derivative of $E$ under forgetting the $C_2$-action and the change of universe $\IRU$. 
\end{prop}
\begin{proof}
Lemma \ref{lem: derivatives and forget} provides a natural isomorphism, $\ind_m^n (i^*F) \cong i^*(\ind_m^nF)$ in $\RE_n$. The change of universe functor preserves natural isomorphisms, hence there is a natural isomorphism 
\[
\IRU (\ind_m^n (i^*F)) \cong \IRU ( i^*(\ind_m^nF)),
\]
which composed with the natural isomorphism, $\ind_m^n (\IRU  i^*F) \cong \IRU  (\ind_m^n (i^*F))$, of Corollary \ref{cor: derivatives and change of universe} applied to $i^* F \in \RE_m$, yields the result. 
\end{proof}

\section{Recovering polynomial functors}\label{section: polynomials}

We now turn our attention to the recovery of unitary polynomial functors. The real strength in the calculus with reality is highlighted by this approach. To compare the polynomial functors from Goodwillie calculus and the polynomial functors from orthogonal calculus, Barnes and Eldred had to first understand how the homogeneous functors compared, and use an induction argument on the polynomial degree, see \cite[\S3]{BE16}. Further conditions were imposed on the functors to directly compare polynomial approximations (see \cite[Proposition 3.4]{BE16}), and this was mirrored in the comparisons between orthogonal and unitary calculus (see \cite[Section 5]{Ta20}). However, we will see that there are direct comparisons between the polynomial functors (see Proposition \ref{thm: recover poly property}) and polynomial approximations (see Theorem \ref{thm: recover polys}) in the calculus with reality and in unitary calculus.

One characterisation of when a functor is $n$--polynomial (see \cite[Proposition 5.2]{We95}) involves the space of enriched natural transformations from the sphere bundle $S\gamma_n(V,-)$, of the $n$--th complement vector bundle, to the functor in question. These spaces of enriched natural transformations admit a description as enriched ends. The following criterion for cofinality of enriched ends is well known to the experts, but we gather it here for the reader's convenience. It will be instrumental in recovering unitary $n$--polynomial functors from $n$--polynomial functors with reality. Given a ($\cC$-enriched) functor $F: \J \to \cD$, and a ``weight'', $W: \J \to \cC$, the limit of $F$, weighted by $W$, denoted $\lim_W F$, is, if it exists, the object in $\cD$ which satisfies the universal property
\[
\cD (D, \sideset{}{_{W}}\lim F) \cong \nat(W(-), \cD(D, F(-))),
\]
natural in $D$, where $\nat(-,-)$ is the space of $\cC$--enriched natural transformations. Similarly to limits, weighted limits have a cofinality statement, which we have re-phrased from Kelly, \cite[Proposition 4.57]{Ke05}.

\begin{prop}[{\cite[Proposition 4.57]{Ke05}}]
Let $u \colon \mathcal{I} \to \J$ be a ($\cC$--enriched) functor,  $W \colon \mathcal{I} \to \cC$ and $W' \colon \J \to \cC$ be ($\cC$--enriched) weights, and $w \colon W u \to W'$ a ($\cC$--enriched) natural transformation. The following are equivalent.
\begin{enumerate}
\item The functor $u : \mathcal{I} \to \J$ is cofinal. That is,for any ($\cC$--enriched) functor $F \colon \J \to \cC$, 
\[
 \sideset{}{_{W'}}\lim F \cong  \sideset{}{_{W}}\lim Fu,
\]
via the ($\cC$--enriched) natural transformation $w \colon Wu \to W'$, either side existing when the other does.
\item The ($\cC$--enriched) natural transformation $w  \colon Wu \to W'$ exhibits $W'$ as the left Kan extension of $W$ along $u$. 
\end{enumerate}
\end{prop}

Applying this to enriched ends, which are examples of enriched weighted limits, yields the following. 

\begin{prop}\label{prop: cofinality of ends}
Let $u : \mathcal{I} \to \J$ be a ($\cC$--enriched) functor. The following are equivalent.
\begin{enumerate}
\item For every ($\cC$--enriched) functor $F: \J^\op \times \J \to \cC$, for which the ($\cC$--enriched) end $\int_{j \in \J} F(j,j) $ exists, there is a natural ($\cC$--enriched) isomorphism 
\[
\int_{j \in \J} F(j,j) \longrightarrow \int_{i \in \mathcal{I}} F(u(i), u(i)).
\]
\item For every $j, j' \in \J$, the canonical map of ($\cC$--enriched) coends, induced by composition
\[
\int^{i \in \mathcal{I}} \Hom_\J(j, u(i)) \times \Hom_\J(u(i), j') \longrightarrow \Hom_\J(j, j'),
\]
is a natural ($\cC$--enriched) isomorphism. 
\end{enumerate}
\end{prop}
\begin{proof}
If $F: \J^\op \times \J \to \cC$ is a ($\cC$--enriched) functor, then the end $\int_{j \in \J} F(j,j)$ is the weighted limit, $\lim_{\Hom_\J} F$, of $F$ weighted by the Hom--functor $\Hom_\J : \J^\op \times \J \to \cC$. The result then follows by cofinality of ($\cC$--enriched) weighted limits, under the weights $W = \Hom_\mathcal{I}$, and $W' = \Hom_\J$, and writing the left Kan extension of $W$ along $u$ as a coend. 
\end{proof}

The relevant example for our purposes is the following. 

\begin{ex}\label{ex: cofinal R to U}
The inclusion functor $j : \CJ_0 \to \UJ_0$ is cofinal, and hence for any functor
\[
F: ({\UJ})^\op \times \UJ \to \T,
\]
for which the end $\int_{V \in \J} F(V,V) $ exists, there is a natural isomorphism, 
\[
\int_{V\in \UJ_0} F(V,V) \longrightarrow \int_{U_\C \in \CJ_0} F(j(U_\C), j(U_\C)).
\]
\end{ex}
\begin{proof}
By Proposition \ref{prop: cofinality of ends}, it suffices to show that for every $V, W \in \UJ_0$ the canonical map of $\T$--enriched coends
\[
\int^{ U_\C \in \CJ_0} \UJ_0(V,j(U_\C)) \wedge \UJ_0(j(U_\C), W) \longrightarrow \UJ_0(V, W).
\]
is a homeomorphism. Every $V \in \UJ_0$ is isometrically isomorphic to an object of $V_\C' \in \CJ_0$, hence the coend of the domain is isomorphic to the coend
\[
\int^{ U_\C \in \CJ_0} \CJ_0(V_\C', U_\C) \wedge \CJ_0(U_\C, W_\C'),
\]
and the result follows. 
\end{proof}

\subsection{Forgetting the $C_2$--action}
Our use of the underlying model structure on $C_2\T$ gives a strong relationship between polynomial functors in the calculus with reality and polynomial functors in unitary calculus indexed on $\C \otimes \R^\infty$. 

\begin{lem}\label{lem: poly and forget}
A functor $F \in \CE_0$ is $n$--polynomial if and only if $i^*F \in \RE_0$ is $n$--polynomial.
\end{lem}
\begin{proof}
A functor $F$ in $\CE_0$ is $n$--polynomial if and only if the map 
\[
F(V) \longrightarrow \underset{U \in \bR_{\leq n+1}}{\holim}~ F(V \oplus U),
\]
is a weak homotopy equivalence in $C_2\T$ for all $V \in \CJ_0$. By our choice of model structure on $C_2\T$, this map is a weak equivalence if and only if, for all $V \in \CJ_0$, the map 
\[
i^*F(V) \longrightarrow i^*\underset{U \in \bR_{\leq n+1}}{\holim}~ F(V \oplus U),
\]
is a weak equivalence. Forgetting the $C_2$--action on the homotopy limit is equivalent to the homotopy limit of the diagram after forgetting the $C_2$--action, hence the above map is a weak equivalence if and only if, for all $V \in \CJ_0$, the map 
\[
i^*F(V) \longrightarrow \underset{U \in \bR_{\leq n+1}}{\holim}~ i^*F(V \oplus U),
\]
is a weak equivalence, that is, if and only if $i^*F$ is $n$--polynomial.
\end{proof}

As a direct corollary we achieve the following. 

\begin{cor}\label{cor: forget and tau}
For $F \in \CE_0$, there is a natural isomorphism $i^* (\tau_nF) \cong \tau_n (i^*F)$.
\end{cor}

This allows us to relate the polynomial approximations for calculus with reality and unitary calculus indexed on the universe $\C \otimes \R^\infty$,  since if two diagrams are term--wise naturally isomorphic, so are the homotopy colimits of the diagrams. 

\begin{lem}\label{lem: forget and poly approx}
For $F \in \CE_0$, there is a natural isomorphism $i^* (T_nF) \cong T_n(i^*F)$ in $\RE_0$. 
\end{lem}
\begin{proof}
The forgetful functor commutes with homotopy colimits, hence the result follows by Corollary \ref{cor: forget and tau}, and the definition of $T_nF$ as the sequential homotopy colimit over iterative applications of $\tau_n$. 
\end{proof}

We can also compare the model structures. 

\begin{prop}\label{lem: forget QA on n-poly}
The adjoint pair
\[
\xymatrix@C+2cm{
n\poly\CE_0 \ar@<-1ex>[r]_{i^*} & n\poly\RE_0 \ar@<-1ex>[l]_{(C_2)_+ \wedge (-)}, \\
}
\]
is a Quillen adjunction between the $n$--polynomial model structures.
\end{prop}
\begin{proof}
Let $f \colon E \to F$ be a weak equivalence in $n\poly\CE_0$, then $T_nf \colon T_nE \to T_nF$ is a levelwise weak equivalence. The right adjoint preserves polynomial approximations by Lemma \ref{lem: forget and poly approx}, hence 
\[
T_ni^*f \colon T_n(i^*E) \to T_n(i^*F),
\]
is a levelwise weak equivalence in $\RE_0$. A similar argument shows that the right adjoint preserves fibrations.
\end{proof}

\subsection{Change of universe}

There are two change of universe functors which we wish to consider. We start with the more straightforward change of universe functor, $\IUR : \E_0^\bU \to \E_0^\bR$, defined via precomposition with the inclusion of categories $j: \J_0^\bR \to \J_0^\bU$.

\begin{lem}\label{lem: agreement of tau_n}
If $F \in \E_0^\mathbf{U}$, then the inclusion of categories $\CJ_0 \hookrightarrow \J_0^\bU$ induces a natural isomorphism between $\IUR  \tau_n^\mathbf{U} F$ and $\tau_n^\mathbf{R} \IUR  F$.
\end{lem}
\begin{proof}
By the unitary version of \cite[Proposition 5.2]{We95}, and  \cite[Proposition 2.9]{Ta20b}, there are natural isomorphisms
\[
\begin{split}
\IUR  (\tau_n^\mathbf{U} F)(V_\C) &\cong \nat_{\E_0^\mathbf{U}}(S\gamma_{n+1}^\mathbf{U}(V_\C, -)_+ , F) \\
\tau_n^\mathbf{R} (\IUR F)(V_\C) &\cong \nat_{\E_0^\mathbf{R}}(S\gamma_{n+1}^\mathbf{R}(V_\C, -)_+ , \IUR F).\\
\end{split}
\]
These spaces of natural transformations are $\T$--enriched ends, given by 
\[
\int_{U \in \J_0^\mathbf{U}} \Map_*(S\gamma_{n+1}^\mathbf{U}(V_\C, U), F(U)),
\]
and
\[
\int_{W_\C \in \J_0^\mathbf{R}} \Map_*(S\gamma_{n+1}^\mathbf{R}(V_\C, W_\C), F(W_\C)),
\]
respectively. Proposition \ref{prop: cofinality of ends} and Example \ref{ex: cofinal R to U} yield a natural isomorphism between these $\T$--enriched ends, which is induced by the inclusion of categories $j : \CJ_0 \to \UJ_0$.
\end{proof}

\begin{cor}\label{cor: poly and change of universe U to R}
 If $F \in \E_0^\mathbf{U}$ is $n$-polynomial, then $\IUR  F$ is $n$-polynomial in $\RE_0$. 
\end{cor}
\begin{proof}
We must show that the map
\[
\IUR F \longrightarrow \tau_n^\mathbf{R} (\IUR F),
\]
is a levelwise weak equivalence. By Lemma~\ref{lem: agreement of tau_n}, this map is equivalent to the map
\[
\IUR ( F \longrightarrow \tau_n^\mathbf{U}F),
\]
hence a levelwise weak equivalence since $F$ is $n$-polynomial, and $\IUR$ preserves levelwise weak equivalences.
\end{proof}

By the construction of the homeomorphism between the homotopy limit $\tau_n F(V)$, and the space of natural transformation, $\nat_{\E_0}(\J(V,-), F)$, see \cite[Proposition 5.2]{We95}, the homeomorphism of Lemma \ref{lem: agreement of tau_n} for varying iterations of $\tau_n$, are compatible, in that, the diagram 
\[
\xymatrix{
\IUR( (\tau_n^\bU)^kF)(V_\C) \ar[r] \ar[d] & \IUR((\tau_n^\bU)^{k+1})F(V_\C) \ar[d] \\
(\tau_n^\bR)^k(\IUR F)(V_\C) \ar[r] & (\tau_n^\bR)^{k+1}(\IUR F)(V_\C)
}
\]
commutes, where the vertical arrows are induced by the homeomorphism of Lemma \ref{lem: agreement of tau_n}.

\begin{prop}\label{thm: I_U^R preserves poly approx}
If $F \in \E_0^\mathbf{U}$, then $T_n^\mathbf{R} \IUR  F$ is naturally isomorphic to $\IUR  T_n^\mathbf{U} F$.
\end{prop}
\begin{proof}
This follows from Lemma  \ref{lem: agreement of tau_n}, and the definition of the $n$-polynomial approximation functor as the sequential homotopy colimit over iterated application of $\tau_n$. 
\end{proof}

This should not be surprising to the reader familiar with the various versions of functor calculus. For example, in unitary calculus, say, indexed on $\C \otimes \R^\infty$, the $1$-polynomial approximation of the functor given by 
\[
V_\C \mapsto \bigslant{\U( \C^\infty \oplus V_\C)}{\U(V_\C)}
\]
is levelwise weakly equivalent to the functor given by 
\[
\Omega Q[(S^{V_\C})_{h\U(1)}]
\]
which is also the $1$-polynomial approximation of the same functor, this time, considered in the calculus with reality, see \cite[Example 8.9]{Ta20b}.

The real strength of our approach is highlighted in the following. In comparison with $\IRU$, the change of universe functor $\IUR$ is significantly more straightforward. Since these functors are inverse equivalences of categories, we can use our results for changing from a calculus indexed on $\C^\infty$ to a calculus indexed on $\C \otimes \R^\infty$ via $\IUR$, to obtain analogous results for the reverse direction, $\IRU$.

\begin{lem}\label{lem: I_R^U and tau_n}
If $F \in \RE_0$, then there is natural isomorphism
\[
\tau_n^\mathbf{U} \IRU  F \longrightarrow \IRU  \tau_n^\mathbf{R} F. 
\]
\end{lem}
\begin{proof}
Lemma \ref{lem: agreement of tau_n} gives a natural isomorphism between $\tau_n^\bR (\IUR E)$ and $\IUR(\tau_n^\bU E)$, for all  unitary functors $E$. In particular, taking $E= \IRU F$, we obtain a natural isomorphism between $\tau_n^\bR (\IUR (\IRU F))$ and $\IUR(\tau_n^\bU (\IRU F))$, for all $F \in \RE_0$. Postcomposing both sides of this natural isomorphism with $\IRU$ and using that $(\IRU, \IUR)$ are adjoint equivalences of categories, yields a natural isomorphism between $\IRU (\tau_n^\bR F)$ and $\tau_n^\bU (\IRU F)$.
\end{proof}

\begin{cor}\label{cor: poly and change of universe R to U}
If $F \in \RE_0$ is $n$--polynomial, then $\IRU F$ is $n$--polynomial.
\end{cor}
\begin{proof}
The argument is analogous to Corollary~\ref{cor: poly and change of universe U to R} using Lemma~\ref{lem: I_R^U and tau_n} in place of Lemma~\ref{lem: agreement of tau_n}.
\end{proof}

We can use a similar argument to give a relation between the change of universe functor $\IRU $ and the polynomial approximation. 

\begin{prop}\label{thm: change of universe and poly approx}
If $F \in \E_0^\mathbf{R}$, then $T_n^\mathbf{U} \IRU  F$ is naturally isomorphic  to $\IRU  T_n^\mathbf{R} F$.
\end{prop}
\begin{proof}
Proposition \ref{thm: I_U^R preserves poly approx} gives a natural isomorphism between $T_n^\bR (\IUR E)$ and $\IUR (T_n^\bU E)$ for all $E \in \E_0^\bU$. Taking $E = \IRU F$, postcomposing both sides of the natural isomorphism with $\IRU$, and using that $(\IRU, \IUR)$ are adjoint equivalences of categories yields the result. 
\end{proof}

Using the fact that the pair $(\IRU , \IUR )$ is an adjoint equivalence of categories we can show that polynomial functors in unitary calculus indexed on $\C^\infty$ are completely characterised by polynomial functors in unitary calculus indexed on $\C \otimes \R^\infty$. 

\begin{prop}\label{thm: poly and change of universe iff}\hfill
\begin{enumerate}
\item $F \in \E_0^\bU$ is $n$--polynomial if and only if $\IUR  F$ is $n$--polynomial.
\item $F \in \RE_0$ is $n$--polynomial if and only if $\IRU  F$ is $n$--polynomial.
\end{enumerate}
\end{prop}
\begin{proof}
One direction of $(1)$ is given in Corollary \ref{cor: poly and change of universe U to R}. For the other direction, let $\IUR F$ be $n$--polynomial. Then by Corollary \ref{cor: poly and change of universe R to U}, $\IRU (\IUR F)$ is $n$--polynomial. Since $(\IRU , \IUR )$ is an adjoint equivalence of categories, this last is naturally isomorphic to $F$, and the result follows. 

One direction of $(2)$ is given in Corollary \ref{cor: poly and change of universe R to U}. In the other direction, let $\IRU  F$ be $n$--polynomial. By Corollary \ref{cor: poly and change of universe U to R}, $\IUR \IRU F$ is $n$--polynomial and the result follows again by the equivalence of categories given by the change of universe functors. 
\end{proof}

\subsection{Complete picture}

We can use the above theorem to give a complete characterisation of polynomial functors in the calculi. 

\begin{prop}\label{thm: recover poly property}
A functor with reality, $F$, is $n$--polynomial if and only if $\IRU  (i^*F)$ is $n$--polynomial in $\E_0^\mathbf{U}$.
\end{prop}
\begin{proof}
By Lemma \ref{lem: poly and forget}, $F$ is $n$--polynomial if and only if $i^*F \in \RE_0$ is $n$--polynomial. It follows by Proposition \ref{thm: poly and change of universe iff}(2) that for $i^*F \in \RE_0$ to be $n$--polynomial, it is necessary and sufficient for $\IRU(i^*F) \in \E_0^\bU$ to be $n$--polynomial. 
\end{proof}

Combining Lemma \ref{lem: forget and poly approx} and Proposition \ref{thm: change of universe and poly approx} we can directly recover the $n$--polynomial approximation functor for unitary calculus from the $n$--polynomial approximation functor for calculus with reality. 

\begin{thm}\label{thm: recover polys}
For $F \in \CE_0$, there is a natural isomorphism between $T_n(\IRU (i^*F))$ and $\IRU (i^*(T_nF))$. 
\end{thm}

We can now consider the $n$--polynomial model structures. We start with the change of universe adjunctions. We should not expect Quillen equivalences between polynomial model structures as the change of universe functors do not reflect levelwise weak equivalences. 

\begin{prop}\label{lem: QA for change of universe and poly}
The adjoint pairs 
\[
\xymatrix@C+2cm{
n\poly\RE_0 \ar@<3ex>[r]|{\IRU } \ar@<-3ex>[r]|{\IRU }  & n\poly\E_0^\mathbf{U} \ar[l]|{\IUR },
}
\]
are Quillen adjunctions between the $n$--polynomial model structures.
\end{prop}
\begin{proof}
We show that $\IUR $ is both left and right Quillen. Since $\IUR $ preserves levelwise weak equivalences, and by Theorem  \ref{thm: I_U^R preserves poly approx},  polynomial approximations, $\IUR $ preserves $T_n$--equivalences. The fact that $\IUR $ is left Quillen follows since it preserves generating cofibrations, and hence projective cofibrations.

To see that $\IUR $ is right Quillen it suffices to show that it preserves fibrant objects, this follows from Corollary \ref{cor: poly and change of universe U to R} since the fibrant objects of the $n$--polynomial model structure are precisely $n$--polynomial functors. 
\end{proof}

The end result of this section is the following commutative diagram 
\[
\xymatrix@C+1cm@R+1cm{
n\poly\CE_0  \ar@<-1.5ex>[d]_{1} \ar@<-1.5ex>[r]_{i^*}	& n\poly\RE_0 \ar@<-1.5ex>[d]_{1} \ar@<3ex>[r]|{\IRU } \ar@<-3ex>[r]|{\IRU }	\ar@<-1.5ex>[l]_{(C_2)_+ \wedge (-)}	& n\poly\E_0^\mathbf{U} \ar@<-1.5ex>[d]_{1} \ar[l]|{\IUR }	 \\
(n-1)\poly\CE_0   \ar@<-1.5ex>[u]_{1}	\ar@<-1.5ex>[r]_{i^*}	&(n-1)\poly\RE_0  \ar@<-1.5ex>[u]_{1} \ar@<3ex>[r]|{\IRU } \ar@<-3ex>[r]|{\IRU } \ar@<-1.5ex>[l]_{(C_2)_+ \wedge (-)}		& (n-1)\poly\E_0^\mathbf{U}  \ar@<-1.5ex>[u]_{1} \ar[l]|{\IUR }	  \\
}
\]
of Quillen adjunctions.

\section{Recovering homogeneous functors}\label{section: homogeneous}
A functor is $n$--homogeneous if it is both $n$--polynomial and has trivial $(n-1)$--polynomial approximation. The homogeneous functors are the building blocks of the calculi, as the $n$--th layer of the Taylor tower (i.e., the homotopy fibre of the map $T_nF \to T_{n-1}F$) is $n$--homogeneous. As these functors are characterised by polynomial properties, we can directly compare homogeneous functors with reality and  homogeneous unitary functors. 

\subsection{Forgetting the $C_2$--action}
We begin the recovery of $n$--homogeneous functors by considering the forgetful functor $i^*: \CE_0 \to \RE_0$.

\begin{lem}\label{lem: homog and forget}
A functor $F \in \CE_0$ is $n$--homogeneous if and only if $i^*F \in \RE_0$ is $n$--homogeneous.
\end{lem}
\begin{proof}
By Lemma \ref{lem: poly and forget} we have that $F$ is $n$--polynomial if and only if $i^*F$ is $n$--polynomial. Moreover, since $T_{n-1}F$ vanishes, we have that  $i^*T_{n-1}F$ also vanishes.  By Lemma \ref{lem: forget and poly approx}, $i^*T_{n-1}F$ is naturally isomorphic to $T_{n-1}(i^*F)$, and the result follows. 
\end{proof}

We can also consider the model categories.

\begin{prop}\label{lem: QA on homog with forget}
The adjoint pair
\[
\xymatrix@C+2cm{
n\homog\CE_0 \ar@<-1ex>[r]_{i^*} & n\homog\RE_0 \ar@<-1ex>[l]_{(C_2)_+ \wedge (-)}, \\
}
\]
is a Quillen adjunction between the $n$--homogeneous model structures.
\end{prop}
\begin{proof}
The right adjoint preserves fibrations by Proposition \ref{lem: forget QA on n-poly}, since the fibrations of the $n$--polynomial and $n$--homogeneous model structures agree. Let $f \colon E \to F$ be a map in the $n$--homogeneous model structure on $\CE_0$. If $f$ is a weak equivalence, then, the map $D_nf \colon D_nE \longrightarrow D_nF$, is a levelwise weak equivalence. The map $i^*D_nf$ fits into a commutative diagram of the form
\[
\begin{tikzcd}[row sep=scriptsize, column sep=scriptsize]
                                         & D_n(i^*E) \arrow[rr] \arrow[dd] &                                                   & T_n(i^*E) \arrow[rr] \arrow[dd] &                                             & T_{n-1}(i^*E) \arrow[dd] \\
i^*D_nE \arrow[dd] \arrow[crossing over]{rr} \arrow[ru] &                                 & i^*T_nE \arrow[crossing over]{rr} \arrow[dd] \arrow[ru, "\cong"] &                                 & i^* T_{n-1}E \arrow[dd] \arrow[ru, "\cong"] &                          \\
                                         & D_n(i^*F) \arrow[rr]            &                                                   & T_n(i^*F) \arrow[rr]            &                                             & T_{n-1}(i^*F)            \\
i^*D_nF \arrow[rr] \arrow[ru]            &                                 & i^*T_nF \arrow[rr] \arrow[ru, "\cong"]            &                                 & i^* T_{n-1}F \arrow[ru, "\cong"]            &                         
\end{tikzcd}
\]
where the indicated maps are natural isomorphisms by Lemma \ref{lem: forget and poly approx}. It follows that the maps $i^*D_nE \to D_n(i^*E)$, and $i^*D_nF \to D_n(i^*F)$ are natural isomorphisms, and hence
\[
D_n(i^*f) \colon D_n(i^*E) \to D_n(i^*F),
\] 
is a levelwise weak equivalence, and hence, the right adjoint preserves weak equivalences. 
\end{proof}

\subsection{Change of universe}

The strong relationship between polynomial functors and the change of universe functors extends to a strong relationship between homogeneous functors and the change of universe functors, as homogeneous functors are defined via polynomial properties.

\begin{lem}\label{lem: homog and change of universe U to R}
If $F \in \E_0^\mathbf{U}$ is $n$--homogeneous, then $\IUR F$ is $n$--homogeneous in $\RE_0$.
\end{lem}
\begin{proof}
The change of universe functors preserve polynomiality (see Proposition \ref{thm: poly and change of universe iff}) hence it suffices to show that the $(n-1)$--polynomial approximation of $\IUR F$ is trivial. Indeed, by Proposition \ref{thm: I_U^R preserves poly approx}, $T_{n-1}^\bR (\IUR F)$ is naturally isomorphic to $\IUR (T_{n-1}^\bU F)$, and the result follows since $T_{n-1}^\bU F$ is trivial, and change of universe preserves levelwise weak equivalences. 
\end{proof}

\begin{lem}\label{lem: homog and change of universe R to U}
If a functor $F \in \RE_0$ is $n$--homogeneous, then $\IRU  F$ is $n$--homogeneous in $\E_0^\mathbf{U}$.
\end{lem}
\begin{proof}
It suffices to show that the $(n-1)$--polynomial approximation of $\IRU F$ is trivial, with follows a similar argument as above. 
\end{proof}

We can hence relate the model structures, this time achieving Quillen equivalences by utilising the Quillen equivalence between the n-homogeneous model structures and the intermediate categories. 

\begin{thm}\label{thm: QE of homog model structures}
The adjoint pairs
\[
\xymatrix@C+2cm{
n\homog\RE_0 \ar@<3ex>[r]|{\IRU } \ar@<-3ex>[r]|{\IRU }  & n\homog\E_0^\mathbf{U} \ar[l]|{\IUR },
}
\]
are Quillen equivalences between the $n$--homogeneous model structures.
\end{thm}
\begin{proof}
We have to show that $\IUR $ is both a left and right Quillen functor as part of a Quillen equivalence. Since $D_n$--equivalences are defined levelwise, $\IUR $ preserves $D_n$--equivalences. The fibrations of the $n$--homogeneous model structure are precisely the fibrations of the $n$--polynomial model structure, which are, by Proposition \ref{lem: QA for change of universe and poly}, preserved by $\IUR$. Hence $\IUR $ is a right Quillen functor. 

Let $f \colon E \to F$ be a cofibration in $n\homog\E_0^\bU$. Lemma \ref{lem: cofibrations of n-homog} characterises $f$ as a projective cofibration and an $(n-1)$--polynomial equivalence, both of which are preserved by $\IUR $. Hence $\IUR $ is a left Quillen functor.

To exhibit that these adjoint pairs are Quillen equivalences we must show that $\IUR $ reflects weak equivalences and both the unit of the $(\IRU, \IUR)$ adjunction, and counit of the $(\IUR, \IRU)$ adjunction, are $n$--homogeneous equivalences. This latter follows since the (co)unit is an isomorphism on the underlying categories, hence, in particular,  a levelwise weak equivalence. For the former, let $f \colon E \to F$ be a map in $n\homog\E_0^\mathbf{U}$, such that, 
\[
D_n^\bR(\IUR f) \colon D_n^\bR(\IUR E) \longrightarrow D_n^\bR(\IUR F),
\]
is a levelwise weak equivalence. Proposition \ref{prop: we in n--homog} gives that the map 
\[
\ind_0^n T_n^\bR(\IUR f) \colon \ind_0^nT_n^\bR(\IUR E) \longrightarrow \ind_0^n T_n^\bR(\IUR F),
\] 
is a levelwise weak equivalence. Combining Corollary \ref{cor: derivatives commute with change of universe} and Proposition \ref{thm: I_U^R preserves poly approx} yields two natural isomorphisms, $\ind_0^nT_n^\bR(\IUR E) \to \IUR (\ind_0^nT_n^\bU E)$, and, $\ind_0^nT_n^\bR (\IUR F) \to \IUR (\ind_0^nT_n^\bU F)$, hence the above map is equivalent to the map
\[
\IUR (\ind_0^n T_n^\bU f) \colon \IUR (\ind_0^nT_n^\bU E) \longrightarrow \IUR (\ind_0^n T_n^\bU F). 
\]
The change of universe functor reflects $n\pi_*$--isomorphisms (see Proposition \ref{lem: change of universe QE on intermediate}), hence the map
\[
\ind_0^n T_n^\bU f \colon \ind_0^nT_n^\bU E \longrightarrow \ind_0^n T_n^\bU F. 
\]
is an $n\pi_*$--isomorphism. The result follows upon noting that the above map is an $n\pi_*$--isomorphism between $n\Omega$--spectra (see \cite[Lemma 5.7]{Ta19}), hence a levelwise weak equivalence, and hence $f: E \to F$ is an $n$--homogeneous equivalence. 
\end{proof}

Using the fact that the pair $(\IRU , \IUR )$ is an adjoint equivalence of categories we can show that homogeneous functors in unitary calculus are completely characterised by homogeneous functors in the intermediate input category $\RE_0$. 

\begin{prop}\label{thm: change of universe and homog functors}\hfill
\begin{enumerate}
\item A functor $F \in \E_0^\bU$ is $n$--homogeneous if and only if $\IUR  F \in \RE_0$ is $n$--homogeneous.
\item A functor $F \in \RE_0$ is $n$--homogeneous if and only if $\IRU  F \in \E_0^\bU$ is $n$--homogeneous.
\end{enumerate}
\end{prop}
\begin{proof}
One direction of $(1)$ is given in Lemma \ref{lem: homog and change of universe R to U}. For the other direction, let $\IUR F$ be $n$--homogeneous. Then by Lemma \ref{lem: homog and change of universe U to R}, $\IRU (\IUR F)$ is $n$--homogeneous. Since $(\IRU , \IUR )$ is an adjoint equivalence of categories, this last is naturally isomorphic to $F$, and the result follows. 

One direction of $(2)$ is given in Lemma \ref{lem: homog and change of universe U to R}. In the other direction, let $\IRU  F$ be $n$--homogeneous. By Lemma \ref{lem: homog and change of universe R to U}, $\IUR \IRU F$ is $n$--homogeneous and the result follows again by the equivalence of categories given by the change of universe functors. 
\end{proof}

\subsection{Complete picture}
We can now give a complete characterisation of the relationship between $n$--homogeneous unitary functors and $n$--homogeneous functors with reality. 

\begin{thm}\label{thm: recovery of homog functors}
A functor $F \in \CE_0$ is $n$--homogeneous if and only if $\IRU(i^*F)$ is $n$--homogeneous in $\E_0^\bU$. 
\end{thm}
\begin{proof}
By Lemma \ref{lem: homog and forget}, $F$ is $n$--polynomial if and only if $i^*F \in \E_0$ is $n$--homogeneous, which by Proposition \ref{thm: change of universe and homog functors} is $n$--homogeneous if and only if $\IRU(i^*F) \in \E_0^\bU$ is $n$--homogeneous. 
\end{proof}

The end result of this section is the following diagram 
\[
\xymatrix@C+1cm@R+1cm{
n\homog\CE_0  \ar@<1ex>[d]^{1} \ar@<-1ex>[r]_{i^*}	& n\homog\RE_0 \ar@<1ex>[d]^{1} \ar@<3ex>[r]|{\IRU } \ar@<-3ex>[r]|{\IRU }	\ar@<-1ex>[l]_{(C_2)_+ \wedge (-)}	& n\homog\E_0^\mathbf{U} \ar@<1ex>[d]^{1} \ar[l]|{\IUR }	 \\
n\poly\CE_0   \ar@<1ex>[u]^{1}	\ar@<-1ex>[r]_{i^*}	&n\poly\RE_0  \ar@<1ex>[u]^{1} \ar@<3ex>[r]|{\IRU } \ar@<-3ex>[r]|{\IRU } \ar@<-1ex>[l]_{(C_2)_+ \wedge (-)}		& n\homog\E_0^\mathbf{U}  \ar@<1ex>[u]^{1} \ar[l]|{\IUR }	  \\
}
\]
of Quillen adjunctions.

\section{Recovering the Taylor tower}\label{section: Taylor towers}

In the last section we saw that if a functor with reality is $n$--homogeneous, the image of the functor in unitary calculus is also $n$--homogeneous. In this section our attention moves to directly comparing the layers of the Taylor tower, and hence the towers themselves. 

\subsection{Forgetting the $C_2$--action}
The case for the forgetful functor $i^* : \CE_0 \to \RE_0$ was contained in the proof of Proposition \ref{lem: QA on homog with forget}. We extract the relevant result here. 

\begin{lem}\label{lem: tower after forgetting}
For $F \in \CE_0$, the Taylor tower associated to $i^*F$ is naturally isomorphic to the image of the Taylor tower of $F$ under the forgetful functor $i^* \colon  \CE_0 \to \RE_0$, that is, 
\[
i^*\Tow(F) \cong \Tow(i^*F).
\]
\end{lem}
\begin{proof}
Proposition \ref{lem: QA on homog with forget} identifies the homotopy fibre sequence 
\[
i^*D_nF \longrightarrow i^*T_nF \longrightarrow i^*T_{n-1}F,
\]
with the homotopy fibre sequence 
\[
D_n(i^*F) \longrightarrow T_n(i^*F) \longrightarrow T_{n-1}(i^*F),
\]
hence the towers are equivalent. 
\end{proof}

\subsection{Change of universe}

\begin{lem}\label{lem: tower and IUR}
For $F \in \E_0^\bU$, the Taylor tower associated to $\IUR  F$ is naturally isomorphic to the image of the Taylor tower of $F$ under the change of universe functor $\IUR $, that is
\[
\IUR(\Tow(F)) \cong \Tow(\IUR(F)).
\] 
\end{lem}
\begin{proof}
We have to identify the fibre sequence 
\[
D_n^\bR(\IUR F) \longrightarrow T_n^\bR(\IUR F) \longrightarrow T_{n-1}^\bR (\IUR F),
\]
with the fibre sequence
\[
\IUR(D_n^\bU F) \longrightarrow \IUR(T_n^\bU F) \longrightarrow \IUR(T_{n-1}^\bU F),
\]
obtained by applying the change of universe functor $\IUR$ to the fibre sequence giving the $n$--th layer of the unitary Taylor tower of $F$. It suffices to exhibit a commutative diagram
\[
\xymatrix@C+1cm{
T_n^\bR(\IUR F) \ar[r]^{r_n^\bR} \ar[d] & T_{n-1}^\bR(\IUR F) \ar[d] \\
\IUR(T_n^\bU F) \ar[r]_{\IUR r_n^\bU} & \IUR(T_{n-1}^\bU F)
}
\]
in which the vertical maps are the natural isomorphisms of Proposition \ref{thm: I_U^R preserves poly approx}. The inclusion $\C^n \hookrightarrow \C^{n+1}$ induces a map on homotopy limits, $(\tau_n^\bU)^k F \to (\tau_{n-1}^\bU)^k F$, which is compatible with the homotopy colimits defining the polynomial approximations, in that, the diagram
\[
\xymatrix@C+1cm{
(\tau_n^\bU)^k F \ar[r] \ar[d]^\rho & (\tau_{n-1}^\bU)^k F \ar[d]^\rho \\
(\tau_n^\bU)^{k+1} F \ar[r] & (\tau_{n-1}^\bU)^{k+1} F \\
}
\]
commutes. Consider the cube
\[
\begin{tikzcd}[row sep=scriptsize, column sep=scriptsize]
                                                        & \IUR((\tau_{n-1}^\bU)^k F) \arrow[rr] \arrow[dd] &                                                  & \IUR((\tau_n^\bU)^k F) \arrow[dd] \\
\IUR((\tau_n^\bU)^k F) \arrow[rr, crossing over] \arrow[ru] \arrow[dd] & {} \arrow[r]                                     & \IUR((\tau_n^\bU)^{k+1} F) \arrow[ru] \arrow[dd] &                                   \\
                                                        & (\tau_{n-1}^\bR)^{k}(\IUR F) \arrow[rr]          &                                                  & (\tau_{n-1}^\bR)^{k+1}(\IUR F)    \\
(\tau_n^\bR)^{k}(\IUR F) \arrow[rr] \arrow[ru]          &                                                  & (\tau_n^\bR)^{k+1}(\IUR F) \arrow[ru]            &                                  
\end{tikzcd}
\]
which commutes. Indeed, the top face is the change of universe functor $\IUR$ applied to the compatibility diagram above; the bottom face is the `with reality' version of the same compatibility diagram applied to $\IUR F$; and, the front and back faces are the commutative squares formed by the compatibility of iterations of $\tau_n$ with the natural isomorphism $\IUR(\tau_n^\bU F) \cong \tau_n^\bR(\IUR F)$, and the natural isomorphism $\IUR(\tau_{n-1}^\bU F) \cong \tau_{n-1}^\bR(\IUR F)$, respectively, see Lemma \ref{lem: agreement of tau_n}. 

It is left to show that the left and right hand side faces commute. The left hand side face can be rewritten as
\[
\xymatrix{
\underset{0 \neq U_1, \dots, U_k \subseteq \C^{n+1}}{\holim}~ F(V_\C \oplus U_1 \oplus \cdots \oplus U_k) \ar[r] \ar[d] & \underset{0 \neq U_1, \dots, U_k \subseteq \C^{n}}{\holim}~ F(V_\C \oplus U_1 \oplus \cdots \oplus U_k)  \ar[d] \\
\underset{0 \neq W_1, \dots, W_k \subseteq \C \otimes \R^{n+1}}{\holim}~ F(V_\C \oplus W_1 \oplus \cdots \oplus W_k) \ar[r] & \underset{0 \neq W_1, \dots, W_k \subseteq \C \otimes \R^{n}}{\holim}~ F(V_\C \oplus W_1 \oplus \cdots \oplus W_k)
}
\]
where the top map is induced by the inclusion $\C^n \hookrightarrow \C^{n+1}$; the bottom map is induced by the inclusion $\C \otimes \R^{n} \hookrightarrow \C \otimes \R^{n+1}$; the left and right hand maps are induced by the inclusion of posets $\{0 \neq U \subseteq \C \otimes \R^{k}\} \to \{0 \neq U \subseteq  \C^{k}\}$ for $k=n+1$, and $k=n$, respectively. It follows that the left hand face commutes since the composites 
\[
\{0 \neq U \subseteq \C \otimes \R^{n}\} \hookrightarrow \{0 \neq U \subseteq  \C \otimes \R^{n+1}\} \hookrightarrow  \{0 \neq U \subseteq  \C^{n+1}\},
\]
and
\[
\{0 \neq U \subseteq \C \otimes \R^{n}\}  \hookrightarrow \{0 \neq U \subseteq  \C^n\} \hookrightarrow  \{0 \neq U \subseteq \C^{n+1}\},
\]
agree. The right hand face of the cube commutes similarly to the left hand face by replacing $k$ by $k+1$. 

Taking homotopy colimits in the horizontal direction of the cube produces the required commutative square 
\[
\xymatrix@C+1cm{
T_n^\bR(\IUR F) \ar[r]^{r_n^\bR} \ar[d]_{\cong} & T_{n-1}^\bR(\IUR F) \ar[d]_{\cong} \\
\IUR(T_n^\bU F) \ar[r]_{\IUR r_n^\bU} & \IUR(T_{n-1}^\bU F)
}
\]
where the vertical maps are the natural isomorphisms of Proposition \ref{thm: I_U^R preserves poly approx}. 
\end{proof}

\begin{lem}\label{lem: tower and IRU}
For $F \in \RE_0$, the Taylor tower associated to $\IRU  F$ is naturally isomorphic to the image of the Taylor tower of $F$ under the change of universe functor $\IRU $, that is, 
\[
\IRU(\Tow(F)) \cong \Tow(\IRU(F)).
\]
\end{lem}
\begin{proof}
By Lemma \ref{lem: tower and IUR}, there is a commutative diagram 
\[
\xymatrix{
T_n^\bR(\IUR\IRU F) \ar[r] \ar[d] & T_{n-1}^\bR(\IUR \IRU F) \ar[d] \\
\IUR(T_n^\bU \IRU F) \ar[r] & \IUR(T_{n-1}^\bU \IRU F)
}
\]
for $\IRU F \in \E_0^\bU$. Applying the change of universe functor induces a commutative diagram
\[
\xymatrix{
\IRU(T_n^\bR(\IUR\IRU F)) \ar[r] \ar[d] & \IRU(T_{n-1}^\bR(\IUR \IRU F)) \ar[d] \\
\IRU(\IUR(T_n^\bU \IRU F)) \ar[r] & \IRU(\IUR(T_{n-1}^\bU \IRU F)).
}
\]
Since $(\IRU, \IUR)$ is an equivalence of categories, we obtain a commutative diagram 
\[
\xymatrix{
\IRU(T_n^\bR F) \ar[r] \ar[d]_\cong & \IRU(T_{n-1}^\bR F) \ar[d]_\cong \\
T_n^\bU (\IRU F) \ar[r] & T_{n-1}^\bU (\IRU F)
}
\]
in which the vertical maps are the natural isomorphisms of Proposition \ref{thm: change of universe and poly approx}. It follows that the induced map on horizontal homotopy fibres is a natural isomorphism. 
\end{proof}

\subsection{Complete picture}
The following result gives the complete comparison between the Taylor towers in calculus with reality and the Taylor towers in unitary calculus. The proof of which is the combination of Lemma \ref{lem: tower after forgetting} and Lemma \ref{lem: tower and IRU}. 

\begin{thm}\label{thm: recover tower}
For $F \in \CE_0$, the Taylor tower associated to $\IRU i^*F$ is naturally isomorphic to the image of the Taylor tower of $F$ under the composition of the forgetful functor $i^* \colon \CE_0 \to \RE_0$ with the change of universe functor $\IRU  \colon \RE_0 \to \E_0^\mathbf{U}$, that is, there is a levelwise weak equivalence 
\[
\IRU(i^*(\Tow(F))) \simeq \Tow(\IRU(i^*F)).
\]
\end{thm}

\section{Recovering the model categories for unitary calculus}\label{section: model cats}

We have seen that the Taylor tower in unitary calculus is completely recovered by the Taylor tower in calculus with reality. To fully complete this comparison, we turn our attention to the model categories for the calculi. The $n$--polynomial and $n$--homogeneous model structures have already been considered, hence it is left to examine the relationship between the various intermediate categories (see Definition \ref{def: intermediate cat}) and the categories of spectra involved in unitary calculus and calculus with reality. 

\subsection{Forgetting the $C_2$--action}
Forgetting the $C_2$--action is an example of a change of group functor, which have been comprehensively studied for categories of $G$--spectra by Mandell and May \cite[V.1.2]{MM02}. We show how these change of group functors interact with the intermediate categories and categories of spectra for the calculi.

\subsubsection{Intermediate categories}
Let $X \in C_2 \ltimes \U(n)\E_n^\bR$, then $X$ is defined by $C_2 \ltimes \U(n)$--equivariant structure maps of the form
\[
X_{U,V} \colon \J_n^\mathbf{R}(U,V) \longrightarrow \Map_*(X(U), X(V)).
\]
Forgetting structure through $i \colon \{e\} \to C_2$, yields a $\U(n)$--equivariant map 
\[
i^*X_{U,V} \colon i^*\J_n^\mathbf{R}(U,V) \to i^*\Map_*(X(U), X(V))= \Map_*(i^*X(U), i^* X(V)).
\]
As such, $i:\{e\} \to C_2$ induces a functor $i^* \colon C_2\ltimes \U(n)\E_n^\mathbf{R} \to \U(n)\E_n^\mathbf{R}$, which has a left adjoint $(C_2)_+ \wedge (-)$, given by 
\[
((C_2)_+ \wedge X)(V) = (C_2)_+ \wedge X(V),
\]
with structure maps 
\[
S^{nW} \wedge (C_2)_+ \wedge X(V) \cong (C_2)_+ \wedge (i^*S^{nW} \wedge X(V)) \to (C_2)_+ \wedge X(V \oplus W),
\]
where the isomorphism follows from \cite[Proposition V.2.3]{MM02}.

\begin{lem}\label{lem: forget and intermediates}
The adjoint pair
\[
\adjunction{(C_2)_+ \wedge (-)}{U(n)\E_n^\mathbf{R}}{C_2 \ltimes \U(n)\E_n^\mathbf{R}}{i^*}
\]
is a Quillen adjunction between the $n$--stable model structures. 
\end{lem}
\begin{proof}
The $n$--homotopy groups of $X\in C_2 \ltimes \U(n)\E_n^\bR$ agree with the $n$--homotopy groups of $i^*X$, hence the right adjoint preserves $n\pi_*$--isomorphisms. Let $f \colon X \to Y$ be a fibration in $C_2 \ltimes \U(n)\E_n^\bR$, i.e. $f \colon X \to Y$ is a levelwise fibration and the square
\[
\xymatrix{
X(V) \ar[r]^{f_V} \ar[d] & Y(V) \ar[d] \\
\Omega^{nW}X(V \oplus W) \ar[r]_{\Omega^{nW}f_V} & \Omega^{nW}Y(V \oplus W)
}
\]
is a homotopy pullback for all $V,W \in \CJ_n$. Levelwise fibrations are those maps $f: X \to Y$, such that, for every $V \in \CJ_n$, $f_V : X(V) \to Y(V)$ is a Serre fibrations in $\T$,  hence $i^*$ preserves levelwise fibrations. Moreover, homotopy pullbacks are computed levelwise in $\T$, hence the right adjoint preserves homotopy pullbacks, hence the square
\[
\xymatrix@C+1cm{
(i^*X)(V) \ar[r]^{i^*f_V} \ar[d] & (i^*Y)(V) \ar[d] \\
\Omega^{nW}(i^*X)(V \oplus W) \ar[r]_{\Omega^{nW}i^*f_V} & \Omega^{nW}(i^*Y)(V \oplus W)
}
\]
is a homotopy pullback, and $i^*$ preserves fibrations in the $n$--stable model structure. 
\end{proof}

\subsubsection{Categories of spectra} Let $H$ be a subgroup of a compact Lie group, $G$. Given a spectrum (in any chosen model) with an action of $G$, we can restrict through the subgroup inclusion $i \colon H \to G$ to given the spectrum an action of $H$.

For a spectrum $\Theta$ with an action of $G$, let $i^* \Theta$ be the spectrum with an action of $H$ formed by forgetting structure through $i$. That is, let $\mathbf{F}$ be the category of $\F$--inner product subspaces of $\F^\infty$ with $\F$--linear isometric isomorphisms. For a spectrum $\Theta$ with $G$--action, the evaluations maps 
\[
\Theta_{U,V}\colon \mathbf{F}(U, V) \to \Map(\Theta(U), \Theta(V)),
\]
are $G$--equivariant. We can apply $i^*$ to this, to give a map which is $H$--equivariant by forgetting structure, 
\[
i^*\Theta_{U,V} \colon i^* \mathbf{F}(U,V) \to i^*\Map(\Theta(U), \Theta(V))= \Map(i^*\Theta(U), i^*\Theta(V)).
\]

The functor $i^*: \s[G]\to\s[H]$, has a left adjoint $G_+ \wedge_{H} - \colon \s[H] \to \s[G]$, given on an object $\Theta$ of $\s[H]$ by 
\[
(G_+\wedge_{H} \Theta)(V) = G_+ \wedge_{H} \Theta(V), 
\]
compare \cite[Proposition V.2.3]{MM02}.

\begin{prop}\label{change of group QA}
The adjoint pair
\[
\adjunction{G_+ \wedge_{H} -}{\s[H]}{\s[G]}{i^*},
\]
is a Quillen adjunction. 
\end{prop}
\begin{proof}
This follows immediately from noting that the $\pi_*$-isomorphisms and fibrations in the stable model structure on $\s[G]$ are defined independently of the group action.
\end{proof}

In particular, we achieve the following. 

\begin{cor}
The adjoint pair
\[
\adjunction{(C_2)_+ \wedge (-)}{\s^\mathbf{O}[\U(n)]}{\s^\mathbf{O}[C_2 \ltimes \U(n)]}{i^*},
\]
is a Quillen adjunction between the stable model structures. 
\end{cor}

A similar dissection of the category of $G$--objects in $\CE_1$ yields the following adjoint pair.

\begin{lem}
The adjoint pair
\[
\adjunction{(C_2)_+ \wedge (-)}{\RE_1[\U(n)]}{\CE_1[\U(n)]}{i^*},
\]
is a Quillen adjunction between the stable model structures. 
\end{lem}
\begin{proof}
The same arguments as in Proposition \ref{change of group QA} show that the right adjoint preserves weak equivalences and fibrations in the stable model structure, and hence is a right Quillen functor. 
\end{proof}

\subsection{Change of universe}
The change of universe for the intermediate categories and the categories of spectra also benefit from work of Mandell and May \cite[\S V.1, \S VI.2]{MM02}, where they construct change of universe functors for orthogonal $G$--spectra and $S$--modules. The following is the $n$--th jet category (see Definition \ref{def: jet categories}) version of \cite[V.1.1]{MM02}. The case $n=1, \mathbb{F}=\R$ is precisely the statement of \cite[V.1.1]{MM02}. 

\begin{lem}
For all $n\geq 0$, and $V,W \in \J_n$ with $\dim_\F(V) =\dim_\F(W)$, the evaluation $G(n)$--map
\[
\J_n(V,W) \wedge X(V) \longrightarrow X(W),
\]
of the functor $X\in G(n)\E_n$ induces a $G(n)$-homeomorphism 
\[
\J_n(V,W) \wedge_{\Aut(V)} X(V) \longrightarrow X(W).
\]
\end{lem}
\begin{proof}
The proof follows from \cite[V.1.1]{MM02} upon noting that $\J_n(V,W) = \bF(V,W)_+$ since $\dim_\F(V) =\dim_\F(W)$ and $\bF(V,W)$ is a free right $\Aut(V)$--space generated by any chosen linear isometric isomorphism. 
\end{proof}

As a corollary, we obtain a characterisation of the domain of the $G(n)$--homeomorphism above. 

\begin{cor}
For all $X \in G(n)\E_n$, and $V,W \in  \J_n$ with $\dim_\F(V) =\dim_\F(W)$, there is a homeomorphism
\[
X(V) \longrightarrow \J_n(V,W) \wedge_{\Aut(V)} X(V),
\]
given by sending $x \in X(V)$ to the equivalence class of $(f,x)$ in $ \J_n(V,W) \wedge_{\Aut(V)} X(V)$, where $f: V \to W$ is a chosen linear isometric isomorphism generating $\bF(V,W)_+ = \J_n(V,W)$. 
\end{cor}

\subsubsection{Intermediate categories} 

The change of universe functors are defined analogously for the intermediate categories as they were for the input categories in Section \ref{section: input functors}. Define a functor $\IRU  : \U(n)\E_n^\mathbf{R} \to \U(n)\E_n^\bU$ by 
\[
(\IRU  X)(V) = \J_n^\bU(\C^n, V) \wedge_{\U(n)} X(\C \otimes \R^n),
\]
for $X \in \U(n)\E_n^\mathbf{R}$ and $V \in \J_n^\mathbf{U}$ with $\dim_\C(V) = n$.  The evaluation map
\[
\J_n^\mathbf{U}(V,W) \wedge (\IRU  X)(V) \to X(W),
\]
is given by the composite
\[
\begin{split}
\J_n^\mathbf{U}(V,W) \wedge (\IRU  X)(V) 	&= 							\J_n^\mathbf{U}(V,W) \wedge \J_n^\bU(\C^n, V) \wedge_{\U(n)} X(\C \otimes \R^n) \\
									&\xrightarrow{\Ev \wedge 1} 		\J_n^\mathbf{U}(\C^n, W) \wedge_{\U(n)} X(\C \otimes \R^n) \\
									&\cong						\J_n^\mathbf{U}(\C^m, W) \wedge_{\U(m)} \J_n^\mathbf{U}(\C^n, \C^m) \wedge_{\U(n)} \wedge X(\C \otimes \R^n) \\
									&\xrightarrow{1 \wedge \Ev}		\J_n^\mathbf{U}(\C^m, W) \wedge_{\U(m)} X(\C \otimes \R^m), \\
\end{split}
\]
where $\dim_\C(W)=m$, and the isomorphism is the inverse to the $\U(n)$--homeomorphism 
\[
\J_n^\mathbf{U}(\C^m, W) \wedge_{\U(m)} \J_n^\mathbf{U}(\C^n, \C^m) \longrightarrow \J_n^\mathbf{U}(\C^n, W),
\]
given by composition. We can define a change of universe in the other direction. The inclusion of categories $j \colon \J_n^\bR \to \J_n^\bU$, defines, via precomposition a functor 
\[
\IUR  \colon \U(n)\E_n^\bU \to \U(n)\E_n^\bR,
\]
given by $(\IUR  X)(V_\C) = X(j(V_\C)) = X(V_\C)$. These change of universe functors are adjoint equivalences of categories. 

\begin{lem}[{\cite[V.1.5]{MM02}}]\label{lem: equivalence of intermediate categories}
There are adjoint pairs,
\[
\xymatrix@C+2cm{
\U(n)\RE_n \ar@<3ex>[r]|{\IRU } \ar@<-3ex>[r]|{\IRU }  & \U(n)\E_n^\mathbf{U} \ar[l]|{\IUR }
}
\]
which are adjoint equivalences of categories. 
\end{lem}

Moreover, the change of universe functors form Quillen equivalences between the respective $n$--stable model structures. 

\begin{prop}\label{lem: change of universe QE on intermediate}
The adjoint pairs
\[
\xymatrix@C+2cm{
\U(n)\RE_n \ar@<3ex>[r]|{\IRU } \ar@<-3ex>[r]|{\IRU }  & \U(n)\E_n^\mathbf{U} \ar[l]|{\IUR }
}
\]
are Quillen equivalences between the $n$--stable model structures.
\end{prop}
\begin{proof}
We show that the change of universe functor $\IUR$ is both left and right Quillen, as part of two Quillen equivalences. Firstly note that for $X \in \U(n)\E_n^\bU$, the $n$--homotopy groups of $X$ and of $\IUR X$ are equal, hence $\IUR$ preserves and reflects weak equivalences. The functor $\IUR$ further preserves $n\Omega$--spectra since for a $n\Omega$--spectrum $X$,
\[
\IUR X(V_\C) = X(V_\C) \simeq \Omega^{nW_\C} X(V_\C \oplus W_\C) = \Omega^{nW} (\IUR X)(V_\C \oplus W_\C). 
\]
It follows that $\IUR$ is a right Quillen functor, and to show that $\IUR$ is a right Quillen functor as part of a Quillen equivalence it suffices to show that the derived unit of the adjunction is an isomorphism, which follows from the fact $(\IRU, \IUR)$ is an adjoint equivalence of categories.

To see that $\IUR$ is a left Quillen functor it suffices, by \cite[2.1.20]{Ho99}, to show that $\IUR$ preserves the generating cofibrations. For $\U(n)\E_n^\bU$, the generating cofibrations are of the form 
\[
\J_n^\bU(V,-) \wedge \U(n)_+ \wedge S^{n-1}_+ \longrightarrow \J_n^\bU(V,-) \wedge \U(n)_+ \wedge D^n_+,
\]
for $V \in \J_n^\bU$, and $n \geq 0$. Applying $\IUR$ and observing that $\IUR$ commutes with smash products of based spaces (see \cite[Theorem V.1.5(iv)]{MM02}), we obtain a map 
\[
\IUR(\J_n^\bU(V,-)) \wedge \U(n)_+ \wedge S^{n-1}_+ \longrightarrow \IUR(\J_n^\bU(V,-)) \wedge \U(n)_+ \wedge D^n_+,
\]
for $V \in \J_n^\bU$, and $n \geq 0$. The space $\IUR(\J_n^\bU(V,-))(W_\C) = \J_n^\bU(V,W_\C)$ is $\U(n)$--equivariantly homeomorphic to $\J_n^\bR(U_\C, W_\C)$ for $U \in \J_n^\mathbf{O}$ with $\dim_\R(U) = \dim_\C(V)$. It follows that for $V \in \J_n^\bU$, and $n \geq 0$, the above map is equivalent (up to $\U(n)$--equivariant homeomorphism) to the map 
\[
\J_n^\bR(U_\C,-) \wedge \U(n)_+ \wedge S^{n-1}_+ \longrightarrow \J_n^\bU(U_\C,-) \wedge \U(n)_+ \wedge D^n_+,
\]
which is a cofibration in $\U(n)\E_n^\bR$. Hence $\IUR$ is a left Quillen functor. To show that the adjoint pair with $\IUR$ left Quillen is a Quillen equivalence, it suffices to show that the derived counit is a $n$--stable equivalence, which follows from $(\IUR, \IRU)$ being an adjoint equivalence of categories. 
\end{proof}

\begin{rem}
In \cite[V.1]{MM02}, Mandell and May do not exhibit their change of universe functor $\I_{\mathcal{V}'}^{\mathcal{V}}$, given by the inclusion of $G$--universes $\mathcal{V} \subset \mathcal{V}'$, as a left Quillen functor. Indeed, $\I_{\mathcal{V}'}^\mathcal{V}$ need not preserve levelwise weak equivalences on fixed--points, nor $\pi_*^H$--isomorphisms for $H \leq G$. 
\end{rem}

There is another point of view on  Proposition \ref{lem: change of universe QE on intermediate}. Given an inclusion of universe $\mathcal{V} \subset \mathcal{V}'$ we can (as in \cite[Theorem V.1.7]{MM02}) define the $\mathcal{V}$--model structure on $\U(n)\E_n^\bU$. For the $\mathcal{V}$--level model structure, define weak equivalences and fibrations by restricting attention to inner product spaces in $\mathcal{V}$; i.e., the $\mathcal{V}$--levelwise weak equivalences and fibrations are created by the forgetful functor $I_{\mathcal{V}'}^\mathcal{V}$. To obtain the $\mathcal{V}$--stable model structure, we let the $\mathcal{V}$--stable equivalences and $\mathcal{V}$--fibrations be created by the forgetful functor $I_{\mathcal{V}'}^\mathcal{V}$. For the case $\mathcal{V}' = \C^\infty$, $\mathcal{V}= \C \otimes \R^\infty$, and the inclusion $\mathcal{V} \subset \mathcal{V}'$ is the inclusion $j \colon \C \otimes \R^\infty \to \C^\infty$, the $\V$--stable model structure on $\U(n)\E_n^\mathbf{R}$ is equivalent to the $n$--stable model structure on $\U(n)\E_n^\bR$, since the calculi on $\E_0^\bR$ has in essence been created through this forgetful functor, and \cite[V.1.7]{MM02}, gives the above result.

\subsubsection{Categories of spectra}
The comparisons between the categories of spectra work similarly to the comparisons of the intermediate categories. This is due to the fact that the intermediate categories are categories of $n\mathbb{S}$--modules and spectra is the special case of $n=1$, i.e., categories of spectra are $\mathbb{S}$--modules in some appropriate category.

\begin{prop}\label{lem: equiv of RE and SU}
The adjoint pairs
\[
\xymatrix@C+2cm{
\RE_1 \ar@<3ex>[r]|{\IRU} \ar@<-3ex>[r]|{\IRU }  & \s^\mathbf{U}\ar[l]|{\IUR},
}
\]
are equivalences of categories, and Quillen equivalences between the stable model structures.
\end{prop}
\begin{proof}
The equivalence of categories follows similar to Lemma \ref{lem: equivalence of intermediate categories}, and the Quillen equivalences follows similarly to Proposition \ref{lem: change of universe QE on intermediate}.
\end{proof}

Specialising to $\U(n)$--objects in each category, we obtain the following result, as a corollary of Proposition \ref{lem: equiv of RE and SU}, together with the fact that the model structures in the categories of $\U(n)$--objects are transferred from the model structures on the underlying categories. 

\begin{cor}
The adjoint pairs
\[
\xymatrix@C+2cm{
\RE_1[\U(n)] \ar@<3ex>[r]|{\IRU } \ar@<-3ex>[r]|{\IRU }  & \s^\mathbf{U}[\U(n)] \ar[l]|{\IUR },
}
\]
are equivalences of categories, and Quillen equivalences between the stable model structures.
\end{cor}

\subsection{Complete picture}

The end result is the following diagram (Figure \ref{fig: model cats for reality and UC}) relating the model categories for calculus with reality, and the model categories for unitary calculus. The remainder of this section is concerned with describing how this diagram commutes. We consider each sub--diagram individually. By Lemma \ref{lem: derivatives and forget}, the square labelled $(3)$ commutes, and by Lemma \ref{lem: restriction and change of universe} and Corollary \ref{cor: derivatives commute with change of universe}, the square labelled $(6)$ commutes. Note that all the vertical Quillen adjunctions, and change of universe adjunctions in Figure \ref{fig: model cats for reality and UC} are Quillen equivalences. 

\begin{figure}[htbp]
\[
\xymatrix@C+1cm@R+1cm{
\os[C_2 \ltimes \U(n)] \ar@{}[dr]|{(1)}  \ar@<-1.5ex>[r]_{i^*} \ar@<-1.5ex>[d]_{L_\psi}   & \os[\U(n)] \ar@{}[dr]|{(4)}  \ar@<1.5ex>[d]^{r^*}  \ar@<-1.5ex>[l]_{(C_2)_+ \wedge (-)}   & \os[\U(n)]  \ar@<1.5ex>[d]^{r^*}  \ar@{=}[l]  \\
\CE_1[\U(n)] \ar@{}[dr]|{(2)}   \ar@<-1.5ex>[r]_{i^*}  \ar@<-1.5ex>[u]_{\psi} \ar@<1.5ex>[d]^{(\xi_n)^*}     & \RE_1[\U(n)] \ar@{}[dr]|{(5)}   \ar@<1.5ex>[u]^{r_!} \ar@<-1.5ex>[l]_{(C_2)_+ \wedge (-)}  \ar@<1.5ex>[d]^{(\overline{\alpha_n})^*}   \ar@<-3ex>[r]|{\IRU }  \ar@<3ex>[r]|{\IRU } & \us[\U(n)]  \ar@<1.5ex>[u]^{r_!}  \ar@<1.5ex>[d]^{(\alpha_n)^*}   \ar[l]|{\IUR } \\
C_2 \ltimes \U(n)\RE_n \ar@{}[dr]|{(3)}   \ar@<-1.5ex>[r]_{i^*}  \ar@<1.5ex>[u]^{(\xi_n)_!}  \ar@<-1.5ex>[d]_{\res_0^n/\U(n)}  & \U(n)\RE_n \ar@{}[dr]|{(6)}   \ar@<-1.5ex>[l]_{(C_2)_+ \wedge (-)}  \ar@<1.5ex>[u]^{(\overline{\alpha_n})_!}  \ar@<-1.5ex>[d]_{\res_0^n/\U(n)}   \ar@<-3ex>[r]|{\IRU }  \ar@<3ex>[r]|{\IRU } & \U(n)\E_n^\mathbf{U}    \ar@<1.5ex>[u]^{(\alpha_n)_!}  \ar@<-1.5ex>[d]_{\res_0^n/\U(n)}   \ar[l]|{\IUR } \\
n\homog\CE_0   \ar@<-1.5ex>[r]_{i^*} \ar@<-1.5ex>[u]_{\ind_0^n\varepsilon^*}  & n\homog\RE_0  \ar@<-1.5ex>[l]_{(C_2)_+ \wedge (-)}  \ar@<-1.5ex>[u]_{\ind_0^n\varepsilon^*}   \ar@<-3ex>[r]|{\IRU }  \ar@<3ex>[r]|{\IRU } & n\homog\E_0^\mathbf{U}  \ar@<-1.5ex>[u]_{\ind_0^n\varepsilon^*}   \ar[l]|{\IUR } \\
}
\]
\caption{Model categories for calculus with reality and unitary calculus}
\label{fig: model cats for reality and UC}
\end{figure}

Realification of complex vector spaces defines a functor $r \colon \J_1^\bR \to \J_1^\mathbf{O}$ given on objects by $r(\C^k) = \R^{2k}$, and on morphisms by $r(f,x) = (f_\R, rx)$ where $f_\R$ is the map $f$ viewed as an $\R$--linear map, and $rx$ is the image of $x$ under the canonical map $ \C \otimes_\C f(V) \to \R \otimes_\R (f_\R)(V_\R)$.  Precomposition with $r$ defines a functor $r^* : \us[\U(n)] \longrightarrow \RE_1[\U(n)].$ For $\Theta \in \us[\U(n)]$, the structure maps of $r^*\Theta$ are given by iterating the structure maps of $\Theta$
\[
S^{2} \wedge (r^*\Theta)(\C^k) \xrightarrow{=} S^{2} \wedge \Theta(\R^{2k}) \xrightarrow{\sigma^2} \Theta(\R^{2k+2}) \xrightarrow{=} (r^*\Theta)(\C^{k+1}),
\]
where $\sigma \colon S^1 \wedge \Theta(\R^k) \longrightarrow \Theta(\R^{k+1})$ is the structure map of $\Theta$. As $r^*$ is defined by precomposition it has a left adjoint, $r_!$ given by the left Kan extension along $r$. As was the case in \cite[Corollary 6.5]{Ta19}, this adjoint pair is a Quillen equivalence.

\begin{prop}
The adjoint pair
\[
\adjunction{r_!}{\RE_1[\U(n)]}{\us[\U(n)]}{r^*},
\]
is a Quillen equivalence. 
\end{prop}

\begin{lem}
The diagram 
labelled $(1)$ in Figure \ref{fig: model cats for reality and UC} commutes up to stable equivalence.
\end{lem}
\begin{proof}
Let $X \in \CE_1[\U(n)]$. We show that $r^*(i^*(\psi(X)))$ is stably equivalent to $(i^*X)$. Indeed, 
\[
\begin{split}
\pi_k( r^*(i^*(\psi(X)))) 	&= 		\underset{q}{\colim}~ \pi_{k +2nq}  (r^*(i^*(\psi(X)))(\C \otimes \R^q)) \\
					&= 		\underset{q}{\colim}~ \pi_{k +2nq}  (i^*(\psi(X))(\R^{2q})) \\
					&= 		\underset{q}{\colim}~ \pi_{k +2nq}  \Map_*(S^{i\R^{2q}}, (i^*X)(\C \otimes \R^{2q})) \\
					&\cong 	\underset{q}{\colim}~ \pi_{k +4nq}  (i^*X)(\C \otimes \R^{2q})) \\
					&\cong	\underset{q}{\colim}~ \pi_{k +2nq}  (i^*X)(\C \otimes \R^{q})) \\
					&=		\pi_k(i^*X).\\ \qedhere 
\end{split}
\]
\end{proof}

There is a topological functor $\overline{\alpha_n} : \CJ_n \to \CJ_1$, which is completely analogous to the functor $\alpha_n : \J_n^\bU \to \J_1^\bU$ of Proposition \ref{QE for UC}, and should be thought of as $\alpha_n$ restricted to only take inputs of the form $V_\C$ for $V \in \J_0^\mathbf{O}$. Precomposition with $\overline{\alpha_n}$ defines the right adjoint of a Quillen equivalence between the intermediate category $\U(n)\RE_n$, and the category $\RE_1[\U(n)]$, of $\U(n)$--objects in $\RE_1$. 

\begin{lem}
The diagram 
labelled $(2)$ in Figure \ref{fig: model cats for reality and UC} commutes up to natural isomorphism. 
\end{lem}
\begin{proof}
Let $X \in \CE_1[\U(n)]$. Then 
\[
(i^*(\xi_n)^*)(X)(V) = i^* X(nV) = (\overline{\alpha_n})^* (i^*X)(V). 
\]
The result then follows immediately. Note that the functor $i^*$ restricts the group actions in a compatible way.  The group $C_2 \ltimes \U(n)$ acts on $X(nV)$ by $X((g\sigma) \otimes V) \circ (g\sigma)_{X(nV)}$, which after forgetting the $C_2$--action (i.e., by letting $C_2$ act via the identity element only), gives a $\U(n)$--action on $i^*X(nV)$ by $(i^*X)(\sigma \otimes V) \circ \sigma_{(i^*X)(nV)}$. This is precisely the action of $\U(n)$ on $(\overline{\alpha_n})^*(i^*X)$. 
\end{proof}

%

\begin{lem}
The diagram
labelled $(4)$ in Figure \ref{fig: model cats for reality and UC} commutes up to natural isomorphism.
\end{lem}
\begin{proof}
Consider the diagram 
\[
\xymatrix{
\J_1^\mathbf{O} \ar@{=}[r]  & \J_1^\mathbf{O} \\
\J_1^\bR \ar[u]^r \ar[r]_{j} & \J_1^\bU \ar[u]_r
}
\]
which commutes on objects and morphisms. Hence there is a natural isomorphism
\[
\IUR  \circ r^* = (r \circ j)^* \cong (r \circ \id)^* = r^*,
\]
and the result follows for the diagram in which $\IUR $ is a right adjoint. To see that the diagram commutes when $\IUR$ is considered as a left adjoint, it suffices to postcompose on both sides of the above natural isomorphism with $\IRU$, to yield that the right adjoints of the required diagram commute up to natural isomorphism, and hence so do the left adjoints.
\end{proof}

\begin{lem}
The diagram
labelled $(5)$ in Figure \ref{fig: model cats for reality and UC} commutes up to natural isomorphism.
\end{lem}
\begin{proof}
Consider the diagram 
\[
\xymatrix{
\J_1^\mathbf{R} \ar[r]^{j}  & \J_1^\mathbf{U} \\
\J_n^\bR \ar[u]^{\overline{\alpha_n}} \ar[r]_{j} & \J_n^\bU \ar[u]_{\alpha_n}
}
\]
which commutes on objects and morphisms. Hence there is a natural isomorphism
\[
\IUR  \circ (\alpha_n)^* =(\alpha_n \circ j)^* \cong (j \circ \overline{\alpha_n})^* = (\overline{\alpha_n})^* \circ \IUR,
\]
and the result follows for the diagram in which $\IUR$ is a right adjoint. It is left to consider the diagram in which $\IUR$ is the left adjoint. Postcomposing both sides of the above natural transformation with $\IRU$ gives a natural isomorphism
\[
(\alpha_n)^* \cong \IRU \circ \IUR  \circ (\alpha_n)^* \cong \IRU \circ (\overline{\alpha_n})^* \circ \IUR,
\]
which after precomposing with $\IRU$ on both sides, yields the result
\[
(\alpha_n)^* \circ \IRU \cong \IRU \circ (\overline{\alpha_n})^* \circ \IUR \circ \IRU \cong \IRU \circ (\overline{\alpha_n})^*. \qedhere
\]
\end{proof}

%

\section{Recovering stable splittings of Arone, Miller and Mitchell--Richter}\label{section:Recovering stable splittings of Arone, Miller and Mitchell--Richter}

One key application of unitary calculus is the stable splitting result of Arone \cite{Ar01} which recovers the stable splitting of Stiefel manfiolds due to Miller, \cite{Mi85}, and verified a conjecture of Mahowald on the stable splitting of the Mitchell--Richter filtration of loops on Stiefel manifolds, see \cite[Theorem 1.2]{Ar01}. In this section we show that the unitary stable splitting of Arone, \cite[Theorem 1.1]{Ar01} is a direct consequence of the analogous result in calculus with reality. We start by proving the ``with reality'' version of \cite[Theorem 1.1]{Ar01}, the proof of which is similar. A similar result is contained in the thesis of Tynan \cite[Chapter 5]{Ty16}, but as remarked in \cite[Remark 2.3]{Ta20b}, Tynan's ``Equivariant Weiss Calculus'' and the calculus with reality are not equivalent. 

\begin{lem}\label{lem: reality stable splitting}
Let $F$ be a functor with reality. Suppose there exists a filtration of $F$ by sub--functors $F_n$, such that $F_0$ is the constant functor at a point, and for all $n \geq 1$, the functor
\[
V_\C \longmapsto F_n(V_\C)/F_{n-1}(V_\C)
\]
is (up to natural weak equivalence) of the form
\[
V_\C \mapsto (X_n \wedge S^{\C^n \otimes V_\C})_{h\U(n)}
\]
where $X_n$ is a based space equipped with an action of $C_2 \ltimes \U(n)$. Then the filtration stably splits, i.e., there is a natural stable equivalence 
\[
F \simeq \bigvee_{n=1}^\infty F_n/F_{n-1}.
\]
\end{lem}
\begin{proof}
By \cite[Example 4.9]{Ta20b}, the functor
\[
V_\C \longmapsto \Omega^\infty[(\Theta \wedge S^{\C^n \otimes V_\C})_{h\U(n)}],
\]
for $\Theta$ a spectrum with an action of $C_2 \ltimes \U(n)$, is $n$--polynomial. It follows that the functor 
\[
V_\C \longmapsto \Omega^\infty\Sigma^\infty[(X_n \wedge S^{\C^n \otimes V_\C})_{h\U(n)}],
\]
is $n$--polynomial, for $X_n$ a based $(C_2 \ltimes \U(n))$--space.

Moreover, the functor $\Omega^\infty \Sigma^{\infty+k} (F_n)$ is $n$--polynomial for any $k \geq 0$. To see this, we argue by induction. The base case $n=0$ is trivial since $F_0 = \ast$. There is a fibre sequence
\[
\Omega^\infty\Sigma^{\infty+k} (F_{n-1}) \longrightarrow \Omega^\infty\Sigma^{\infty +k}(F_n) \longrightarrow \Omega^\infty\Sigma^{\infty +k }(F_n/F_{n-1}),
\]
which after rotation gives a fibre sequence 
\[
\Omega^\infty\Sigma^{\infty+k} (F_{n}) \longrightarrow \Omega^\infty\Sigma^{\infty +k}(F_n/F_{n-1}) \longrightarrow \Omega^\infty\Sigma^{\infty +k +1}(F_{n-1}).
\]
By the induction hypothesis the right hand term is $n$--polynomial, and the middle term is $n$--polynomial since the quotient $F_n/F_{n-1}$ is (up to weak equivalence) of the form 
\[
V_\C \mapsto (X_n \wedge S^{\C^n \otimes V_\C})_{h\U(n)}.
\]
It follows by \cite[Lemma 4.7]{Ta20b}, that $\Omega^\infty\Sigma^{\infty+k} (F_{n})$ is $n$--polynomial. 

Consider the commutative diagram 
\[
\xymatrix{
\Omega^\infty\Sigma^\infty (F_{n-1}) \ar[r] \ar[d] & \Omega^\infty\Sigma^\infty(F_n) \ar[r] \ar[d] & \Omega^\infty\Sigma^\infty(F_n/F_{n-1}) \ar[d] \\
T_{n-1}\Omega^\infty\Sigma^\infty(F_{n-1}) \ar[r] & T_{n-1}\Omega^\infty\Sigma^\infty(F_n) \ar[r] & \Omega^\infty\Sigma^\infty(F_n/F_{n-1}).\\
}
\]
in which the rows are fibre sequences, and the bottom right hand functor is levelwise weakly contractible, since $F_n/F_{n-1}$ is $n$--homogeneous. It follows that the lower left hand map is a levelwise weak equivalence. Since $\Omega^\infty\Sigma^\infty (F_{n-1})$ is $(n-1)$--polynomial, the left hand vertical map is also a levelwise weak equivalence, hence the top fibre sequence splits, 
\[
\Sigma^\infty (F_{n}) \simeq \Sigma^\infty (F_{n-1}) \vee \Sigma^\infty (F_n/F_{n-1}) \simeq \Sigma^\infty (F_{n-1} \vee F_n/F_{n-1}).
\]
Arguing via induction on the degree of the filtration yields the result. 
\end{proof}

As corollaries we obtain underlying $C_2$--equivariant stable splittings of Stiefel manifolds and loops on Stiefel manifolds, where underlying means we consider the stable splitting in spectra with an action of $C_2$, rather than `genuine' $C_2$--spectra. For a `genuine' version we would require a `genuine' equivariant version of orthogonal calculus, which to the author's knowledge, does not exist.

\begin{ex}[Miller splitting]\label{ex: miller splitting}
Let $V,W \in \CJ_0$ with $\dim_\C(W) \geq 1$. The Miller filtration on $\CJ_0(V, V\oplus W)$ stably splits. That is, there is a stable equivalence
\[
\CJ_0(V, V \oplus W) \simeq \bigvee_{n=1}^\infty R^n(V;W)/R^{n-1}(V;W).
\]
\end{ex}
\begin{proof}
A linear isometry $f \in \CJ_0(V,V \oplus W)$ may be written as $f= f_1 + f_2$, with $f_1 \in \CJ_0(V, V)$ and $f_2 \in \CJ_0(V,W)$. The terms of the Miller filtration are then defined as
\[
R^n(V;W) = \{ f_1 + f_2 \in \CJ_0(V, V\oplus W) \mid \dim (\ker(f_1 - \id)^\perp) \leq n\},
\]
for all $0 \leq n \leq \dim(V)$. These spaces inherit a $C_2$--action from the $C_2$--action on $\CJ_0(V, V \oplus W)$ by conjugation by complex conjugation. The construction of the filtration is functorial in $W$, hence $R^n(V;-)$ is an object of $\CE_0$. The quotients of the filtration are given by Thom spaces of vector bundles defined over Grassmannian manifolds, and by \cite[p.1208]{Ar01}, may be written as 
\[
R^n(V;W) / R^{n-1}(V;W) \cong (S^{\mathrm{Ad}_{C_2 \ltimes \U(n)}} \wedge \CJ_0(\C^n, V) \wedge S^{nW})_{h\U(n)},
\]
up to $C_2$--equivariant homeomorphism, where $\mathrm{Ad}_{C_2 \ltimes \U(n)}$ is the adjoint representation of $C_2 \ltimes \U(n)$. As such, Lemma \ref{lem: reality stable splitting} implies the stable splitting.
\end{proof}

\begin{rem}
It is also possible to use the adjoint representation of $\U(n)$ with the canonical $C_2 \ltimes \U(n)$--action in place of the adjoint representation of $C_2 \ltimes \U(n)$ since the adjoint representation of a Lie group $G$ is given by  the Lie algebra of $G$, which is completely determined by the connected component of $G$, hence the Lie algebra of $C_2 \ltimes \U(n)$ and the Lie algebra of $\U(n)$ are the same, and their respective $C_2 \ltimes \U(n)$ actions coincide. 
\end{rem}

To discuss the Mitchell--Richter splitting in this setting we have to clearly set out the framework, which is taken from \cite[\S2]{CrabbLoops}. This example was also discussed by Tynan in \cite[Chapter 4, Chapter 5]{Ty16} in his ``Equivariant Weiss Calculus'' framework. Identify $S^1$ with the unit sphere in $\C$ and set 
\[
\Omega_n\U(V) = \{ g \in \Omega\U(V) \mid \mathrm{det} \circ g : S^1 \to S^1 \ \text{has degree} \ n \},
\]
where $\mathrm{det}: \U(V) \to S^1$ is the determinant. Given an $n$--dimensional subspace $U$ of $V$, we define an element $\rho_U : \Omega\U(V)$ by $\rho_U(z) = \pi_Uz + (\id - \pi_U),$ where $z : S^1 \hookrightarrow \C$ is the inclusion, and $\pi_U$ is the orthogonal projection onto $U$. We can define the subspace $S_n(V)$ of $\Omega_n\U(V)$ as the image of the map
\[
\begin{split}
\C P(V) \times \cdots \times \C P(V) &\longrightarrow \Omega_n\U(V) \\
(L_1, \cdots L_n) &\longmapsto \rho_{L_1}\cdots \rho_{L_n},
\end{split}
\]
where the product of the $\rho_L$'s is given by the topological group structure on $\Omega\U(V)$ induced by point--wise multiplication in $\U(V)$. 

\begin{ex}[Mitchell--Richter splitting]\label{mitchell--richter splitting}
Let $V,W \in \CJ_0$, with $\dim_\C(W) \geq 1$. The Mitchell--Richter filtration on $\Omega\CJ_0(V, V\oplus W)$ stably splits. That is, there is a stable equivalence
\[
\Omega\CJ_0(V, V \oplus W) \simeq \bigvee_{n=1}^\infty S^n(V;W)/S^{n-1}(V;W).
\]
\end{ex}
\begin{proof}
The space $\Omega\CJ_0(V, V \oplus W)$, or rather, an algebraic loop space that is weakly equivalent to $\Omega\CJ_0(V, V \oplus W)$ when $\dim(W) \geq 1$, (for detailed see \cite[Chapter 4]{Ty16} or \cite[p.50]{CrabbLoops}) is filtered by subspaces $S^n(V;W)$ for all $n \geq 0$. The subspace $S^n(V;W)$ is the image of a subspace $S_n(V\oplus W)$ under the quotient map $q: \Omega\U(V\oplus W) \mapsto \Omega\CJ_0(V, V \oplus W)$. As in \cite[p.1208]{Ar01} or \cite[\S4.5]{Ty16}, the spaces $S^n(V;W)$ are a well--defined filtration, with quotient $S^n(V;W)/S^{n-1}(V;W)$ homotopy equivalent (in $C_2$--spaces) to the space 
\[
(\tilde{S}_n(V)_+ \wedge S^{nW})_{h\U(n)}
\]
where $\tilde{S}_n(V)$ is the pullback
\[
\xymatrix{
\tilde{S}_n(V) \ar[r] \ar[d] & \ar[d] \CJ(\C^n,V) \\
S_n(V) \ar[r] & \CJ(\C^n, V)/\U(k).  \\
}
\]
The result then follows by Theorem \ref{lem: reality stable splitting}.
\end{proof}

We now show that Lemma \ref{lem: reality stable splitting} implies the unitary stable splitting of Arone, \cite[Theorem 1.1]{Ar01}, and hence Example \ref{ex: miller splitting} and Example \ref{mitchell--richter splitting} imply the classical Miller and Mitchell--Richter splittings respectively. 

\begin{thm}\label{recovering of splitting}
Let $F$ be a functor with reality. Suppose there exists a filtration of $F$ by sub--functors $F_n$, such that $F_0$ is the constant functor at a point, and for all $n \geq 1$, the functor
\[
V_\C \longmapsto F_n(V_\C)/F_{n-1}(V_\C),
\]
is (up to natural levelwise weak equivalence) of the form
\[
V_\C \longmapsto (X_n \wedge S^{\C^n \otimes V_\C})_{h\U(n)},
\]
where $X_n$ is a based space equipped with an action of $C_2 \ltimes \U(n)$. Then the functor $\IRU(i^*F)$ is a unitary functor with a filtration by sub--functors $\IRU(i^*F)_n$, such that $\IRU(i^*F)_0$ is the constant functor at a point, and for all $n \geq 1$, the functor,
\[
W \longmapsto \IRU(i^*F)_n(W)/\IRU(i^*F)_{n-1}(W),
\]
where $W \in \J_0^\bU$ is (up to natural levelwise weak equivalence) of the form
\[
W \mapsto (i^*X_n \wedge S^{\C^n \otimes W})_{h\U(n)}.
\]
Hence $\IRU(i^*F)$ stably splits in the sense of \cite[Theorem 1.1]{Ar01}.
\end{thm}
\begin{proof}
Define the sub--functor $\IRU(i^*F)_n = \IRU(i^*F_n)$. The collection $\{\IRU(i^F)_n\}$ form a filtration of $\IRU(i^*F)$ with $\IRU(i^*F)_0 = \IRU(i^*F_0)=\ast$. Given that the quotient $F_n/F_{n-1}$ are (up to levelwise weak equivalence) of the form 
\[
V_\C \longmapsto (X_n \wedge S^{\C^n\otimes V_\C})_{h\U(n)},
\]
it is left to show that the quotient $\IRU(i^*F)_n/\IRU(i^*F)_{n-1}$ is of the form 
\[
W \mapsto (Y_n \wedge S^{\C ^n \otimes W})_{h\U(n)},
\]
for $Y_n$ a $\U(n)$--space, and $W \in \J_0^\mathbf{U}$. There is a homotopy cofiber sequence 
\[
F_{n-1} \longrightarrow F_n \longrightarrow F_n/F_{n-1},
\]
which induces a homotopy cofibre sequence
\[
\IRU(i^*F)_{n-1} \longrightarrow \IRU(i^*F)_n \longrightarrow \IRU(i^*F)_n/\IRU(i^*F)_{n-1},
\]
hence $\IRU(i^*F)_n/\IRU(i^*F)_{n-1} \simeq \IRU(i^*(F_n/F_{n-1}))$. It follows from Theorem \ref{thm: recover tower}, and the characterisation of $n$--homogeneous functors, Proposition \ref{homog characterisation}, that  $\IRU(i^*F)_n/\IRU(i^*F)_{n-1}$ is (up to natural weak equivalence) of the form, 
\[
W \mapsto (i^*X_n \wedge S^{nW})_{h\U(n)}. \qedhere
\]
\end{proof}

\bibliography{references}

\begin{thebibliography}{MMSS01}

\bibitem[Aro01]{Ar01}
G.~Arone.
\newblock The {M}itchell-{R}ichter filtration of loops on {S}tiefel manifolds
  stably splits.
\newblock {\em Proc. Amer. Math. Soc.}, 129(4):1207--1211, 2001.

\bibitem[Aro02]{Ar02}
G.~Arone.
\newblock The {W}eiss derivatives of {$B{ O}(-)$} and {$B{ U}(-)$}.
\newblock {\em Topology}, 41(3):451--481, 2002.

\bibitem[Ati66]{At66}
M.~F. Atiyah.
\newblock {K -Theory and Reality }.
\newblock {\em The Quarterly Journal of Mathematics}, 17(1):367--386, 01 1966.

\bibitem[BE16]{BE16}
D.~Barnes and R.~Eldred.
\newblock Comparing the orthogonal and homotopy functor calculi.
\newblock {\em J. Pure Appl. Algebra}, 220(11):3650--3675, 2016.

\bibitem[Beh12]{BehrensEHP}
Mark Behrens.
\newblock The {G}oodwillie tower and the {EHP} sequence.
\newblock {\em Mem. Amer. Math. Soc.}, 218(1026):xii+90, 2012.

\bibitem[BO13]{BO13}
D.~Barnes and P.~Oman.
\newblock Model categories for orthogonal calculus.
\newblock {\em Algebr. Geom. Topol.}, 13(2):959--999, 2013.

\bibitem[Cra87]{CrabbLoops}
M.~C. Crabb.
\newblock On the stable splitting of {${\rm U}(n)$} and {$\Omega {\rm U}(n)$}.
\newblock In {\em Algebraic topology, {B}arcelona, 1986}, volume 1298 of {\em
  Lecture Notes in Math.}, pages 35--53. Springer, Berlin, 1987.

\bibitem[Hov99]{Ho99}
M.~Hovey.
\newblock {\em Model categories}, volume~63 of {\em Mathematical Surveys and
  Monographs}.
\newblock American Mathematical Society, Providence, RI, 1999.

\bibitem[Kel05]{Ke05}
G.~M. Kelly.
\newblock Basic concepts of enriched category theory.
\newblock {\em Repr. Theory Appl. Categ.}, 10:vi+137, 2005.
\newblock Reprint of the 1982 original [Cambridge Univ. Press, Cambridge;
  MR0651714].

\bibitem[Mil85]{Mi85}
H.~Miller.
\newblock Stable splittings of {S}tiefel manifolds.
\newblock {\em Topology}, 24(4):411--419, 1985.

\bibitem[MM02]{MM02}
M.~A. Mandell and J.~P. May.
\newblock Equivariant orthogonal spectra and {$S$}-modules.
\newblock {\em Mem. Amer. Math. Soc.}, 159(755):x+108, 2002.

\bibitem[MMSS01]{MMSS01}
M.~A. Mandell, J.~P. May, S.~Schwede, and B.~Shipley.
\newblock Model categories of diagram spectra.
\newblock {\em Proc. London Math. Soc. (3)}, 82(2):441--512, 2001.

\bibitem[Sch19]{Sch19}
S.~Schwede.
\newblock Lectures on equivariant stable homotopy theory.
\newblock \url{http://www.math.uni-bonn.de/people/schwede/SymSpec-v3.pdf},
  2019.

\bibitem[Tag20]{Taggartthesis}
N.~Taggart.
\newblock {\em Beyond orthogonal calculus: the unitary and Real cases}.
\newblock PhD thesis, 2020.
\newblock Thesis (Ph.D.)--Queen's University Belfast.

\bibitem[Tag21]{Ta20}
N.~Taggart.
\newblock Comparing the orthogonal and unitary functor calculi.
\newblock {\em Homology Homotopy Appl.}, 23(2):227--256, 2021.

\bibitem[Tag22a]{Ta19}
N.~Taggart.
\newblock Unitary calculus: model categories and convergence.
\newblock {\em J. Homotopy Relat. Struct.}, 17(3):419--462, 2022.

\bibitem[Tag22b]{Ta20b}
N.~Taggart.
\newblock Unitary functor calculus with reality.
\newblock {\em Glasg. Math. J.}, 64(1):197--230, 2022.

\bibitem[Tyn16]{Ty16}
P.~D. Tynan.
\newblock {\em Equivariant {W}eiss {C}alculus and {L}oop {S}paces of {S}tiefel
  {M}anifolds}.
\newblock ProQuest LLC, Ann Arbor, MI, 2016.
\newblock Thesis (Ph.D.)--Harvard University.

\bibitem[Wei95]{We95}
M.~Weiss.
\newblock Orthogonal calculus.
\newblock {\em Trans. Amer. Math. Soc.}, 347(10):3743--3796, 1995.

\end{thebibliography}
\bibliographystyle{alpha}
\end{document}